\documentclass[12pt, reqno]{amsart}      
\usepackage{amsmath,amssymb,amsfonts,amsthm,amscd,cite}
\usepackage{fancyvrb, hyperref, fancyhdr}
\usepackage{paralist}
\usepackage[dvipsnames]{xcolor}
\usepackage{graphicx}
\usepackage{a4wide}
\usepackage{enumerate, verbatim, stmaryrd}
\usepackage{color,booktabs, float}
\usepackage{comment}
\usepackage[shortlabels]{enumitem}
\newtheorem{theorem}{Theorem}[section]
\newtheorem{proposition}[theorem]{Proposition}

\newtheorem{lemma}[theorem]{Lemma}

\theoremstyle{definition}
\newtheorem{remark}[theorem]{Remark}

\newtheorem{Note}[theorem]{Note}

\newtheorem{assumption}{Assumption}[section]

\usepackage{hyperref}
\usepackage{mathrsfs}
\def\div{\mathop{\mathrm{div}}\nolimits}
\def\R{\mathbb{R}}

\def\U{\tilde{U}}

\def\L{\mathscr{L}}

\newcommand{\csol}{\mathcal{C}l}
\newcommand{\s}{g}
\newcommand{\vecfield}{\mathcal G}
\newcommand{\pathrho}{\{\rho_{\cdot}\}}
\newcommand{\Prd}{\mathcal{P}(\R^d)}
\newcommand{\N}{\mathbb N}
\newcommand{\nup}{^{(n)}}
\newcommand{\tlh}{\tilde{\mathcal{L}}_H}
\newcommand{\rxu}{R_x^{\tilde U}}
\newcommand{\lan}{\langle}
\newcommand{\vd}{\mathbf{d}}
\newcommand{\ran}{\rangle}
\newcommand{\be}{\begin{equation}}
	\newcommand{\ee}{\end{equation}}
\newcommand{\pa}{\partial}
\newcommand{\lv}{\left\vert}
\newcommand{\rv}{\right\vert}
\newcommand{\cL}{\mathcal L}
\newcommand{\cS}{\mathcal S}
\newcommand{\cH}{\mathcal H}
\newcommand{\ltm}{L^2_{\mu}}
\newcommand{\dense}{\mathcal{V}}
\newcommand{\rxv}{R_x^{\tilde U}v}

\newcommand{\cI}{\mathcal I }
\def\S{\mathcal{S}}
\newcommand{\cZ}{\mathcal Z}
\newcommand{\dd}{,{\dots} , }
\newcommand{\cD}{\mathcal{D}}
\newcommand{\psis}{\psi^{\star}}
\newcommand{\psisj}{\hat{\psi}^{\star}}

\newcommand{\cP}{\mathcal P}
\newcommand{\prho}{_{\rho}}
\newcommand{\trans}{\mathcal F}
\newcommand{\mko}{\mathfrak{L}}
\newcommand{\hV}{\hat{V}}
\newtheorem{hypothesis}[theorem]{Hypothesis}

\def\div{\mathop{\mathrm{div}}\nolimits}
\def\R{\mathbb{R}}

\def\U{\tilde{U}}

\begin{document}
	\title{Non-reversible processes: GENERIC, Hypocoercivity and fluctuations. }
	
	\author{M. H. Duong and M. Ottobre}
	
	\address{\noindent \textsc{Manh Hong Duong, School of Mathematics, University of Birmingham, Birmingham B15 2TT, UK}} 
	\email{h.duong@bham.ac.uk}
	
	\address{\noindent \textsc{Michela Ottobre, Maxwell Institute for Mathematical Sciences, Department of Mathematics, Heriot-Watt University, Edinburgh EH14 4AS, UK}} 
	\email{m.ottobre@hw.ac.uk}
	
	\maketitle
	
	\begin{abstract}
		We consider two approaches to study non-reversible Markov  processes, namely the Hypocoercivity Theory (HT) and GENERIC (General Equations for Non-Equilibrium Reversible-Irreversible Coupling); the basic idea behind both of them  is to  split  the process into a reversible component and a non-reversible one, and then quantify the way in which they interact.  We compare such theories and  provide explicit formulas to pass from one formulation to the other; as a bi-product we give a  simple proof  of the link between reversibility of the dynamics and gradient flow structure of the associated  Fokker-Planck equation. We do this both for linear Markov processes and for a class of nonlinear Markov process as well. We then characterize the structure of the Large Deviation functional of generalised-reversible processes; this is a class  of non-reversible processes of large relevance in applications. Finally, we show how our results apply to two classes of Markov processes, namely non-reversible diffusion processes and a class of Piecewise Deterministic Markov Processes (PDMPs), which have recently attracted the attention of the statistical sampling community. In particular, for the PDMPs we consider we prove entropy decay. 
		\vspace{5pt}
		\vspace{5pt}
		\\
		{\sc AMS Classification (MSC 2020). 35Q82, 35Q84, 60J25, 60F10, 60H30,  82B35, 82C31} 
	\end{abstract}
	\date{}

	\section{Introduction}
	The study of non-reversible Markov  processes has attracted the attention of the mathematics and physics communities for several decades; while significant progress and more understanding has been brought by a large body of research, it is fair to say that several important questions regarding non-reversibility still remain unanswered. Non-reversible Markov  processes are fundamental models in  non-equilibrium statistical mechanics, for example in the study of so-called {\em open systems} \cite{Penrose}, see e.g. the vast literature on heat baths \cite{Hairer, bath}, and  in collisional kinetic theory, prominently in connection with the study of the Boltzmann equation, in its many forms and simplifications \cite{Pulvirenti, Villani, Villanicollisional}.  More recently, the observation that non-reversible processes might enjoy favourable properties in terms of speed of convergence to equilibrium, has attracted also  the statistical sampling community to further investigations in this field,  so that  the applications-driven need to produce increasingly high performing algorithms has  played an important role in pushing forward the theory of non-reversible processes \cite{AndDob, AndLiv, Ottobre, DuncPavl}, with particular reference to the class of {\em Piecewise Deterministic Markov Process} (PDMPs) \cite{Davis, DoucetVanetti, FearRos, Faggionato2008, Faggionato2009, Faggionato2012}.  
	
	In this paper we use the word {\em reversible} to refer to stochastic processes which are  {\em time-reversible}, i.e. for any $T>0$ the process $\{X_t\}_{t \in [0,T]}$  and its time-reversed $\{X_{T-t}\}_{t \in [0,T]}$  have the same  distributions (on the space of continuous paths); because of this time-symmetry, reversible processes give rise to  Partial Differential equations (PDEs) associated with  symmetric (i.e. formally self-adjoint) operators \cite{Pavliotis}.    However we point out that  the term ``reversible" is used with very different  meanings throughout the literature, and certainly across the strands of research that we use in this paper; so,  to avoid confusion,  we clarify matters in Subsection \ref{subsec:reversibility}.

	The theory of reversible processes is  by far more settled than its non-reversible counterpart, and this is true both of probabilistic and of analytic/functional analytic approaches.  Nonetheless conceptual frameworks  for the analysis and modelling of non-reversible phenomena do exist; in this paper we consider two of them, the Hypocoercivity Theory (HT), initiated by Herau \cite{Herau} and then made systematic by C. Villani \cite{Villani}, and GENERIC (General Equations for non-Equilibrium Reversible-Irreversible Coupling), whose first complete exposition can be found in \cite{Ottinger}. 
	The premise and purpose of these two theories is quite different: the HT is from its inception a functional analytic theory aimed at studying exponentially fast convergence to equilibrium for non-reversible processes; GENERIC was born in the physics and engineering community,  as a  framework to help  applied scientists  model reversible (dissipative) and non-reversible (conservative) contributions to a given dynamics, and only  later evolved into a more mathematised theory, by hands of \cite{kpelmath, MielkePeletierRenger2014, Mielke2011}, and references therein, which also significantly pushed it forward.   One of the main purposes of this paper is to observe that the formulations of such theories, which can look quite different from the outset,  are indeed complementary and substantially equivalent, in the sense that we can provide formulas to pass from one formulation to the other, and we do so in Section \ref{sec:3} and Section \ref{section:decomposition};  we also give examples to show  how this fact can be exploited in practice. 
	
	To give more background on these two theories  let us start with a toy example. 
	Consider the  second order Langevin equation, that is, the evolution
	\be\label{Langevin}
	dx_t= v_t dt, \qquad 
	dv_t =- \pa_x V(x_t) dt-v_t dt+  \sqrt{2} dW_t\, ,
	\ee
	where $(x_t, v_t) \in \R^2$ (for simplicity),  $W_t$ is one dimensional standard Brownian motion,  the potential $V:\R \rightarrow \R$ is smooth,  confining (i.e. $V(x) \rightarrow \infty$ as $|x| \rightarrow \infty$) and grows at least quadratically at infinity.  Under these conditions the process \eqref{Langevin} admits a unique invariant measure, namely the measure with density 
	$\mu:\R^2 \rightarrow \R$ given by
	\be\label{target}
	\mu(x, v) = \frac{1}{Z}e^{-V(x)} e^{-\frac{v^2}{2}}\,,
	\ee
	where $Z$ is a normalization constant.
	To the SDE \eqref{Langevin}  one can associate two PDEs, namely the Kolmogorov equation (KE), describing the  evolution of so-called observables, i.e. of  quantities of the form $\mathbb {E}(f(x_t, v_t)\vert (x_0, v_0)=(x,v))$, for any given $f \in C_b(\R^2)$,\footnote{Here and throughout $C_b(\R^d)$ denotes the set of functions $f: \R^d \rightarrow \R$ which are continuous and bounded}  and the Fokker-Planck (FP) equation, describing the evolution of the law of the process; in the case of the diffusion \eqref{Langevin} such equations take, respectively, the form
	\be\label{kolmexample}
	\pa_tu_t(x,v) = \cL u_t : =v\pa_xu_t- \pa_x V(x)\pa_v u_t - v\pa_v u_t+ \pa_v^2u_t \,, \quad u_0(x,v)=f(x,v)\, ,
	\ee
	and
	\be\label{FPLangevin}
	\pa_t\rho_t(x,v) =\cL' \rho_t := - v \pa_x \rho_t+ \pa_x V(x) \pa_v \rho_t+ \pa_v(v \rho_t) + \pa_v^2 \rho_t \, , \quad \rho_0 = \mathrm{Law}(x_0, v_0), 
	\ee
	respectively, where  $\cL$ is the operator defined on smooth functions $g:\R^2 \rightarrow \R$ as 
	$$
	\cL g := v\pa_x g  - \pa_x V(x)\pa_v g - v\pa_v g+ \pa_v^2 g \, , 
	$$
	usually referred to as the {\em generator} of the system, 
	and $'$ denotes (formal) adjoint  in $L^2(\R^2):=\{f: \R^2 \rightarrow \R : \int_{\R^2}|f(x,v)|^2dxdv < \infty\}$; the operator $\cL'$ is also called the {\em Fokker-Planck operator} (we will give more context on these equations in Section \ref{Sec2}). 
	
	The evolution \eqref{Langevin} can be split into a Hamiltonian component, namely
	$$
	dx_t= v_t dt, \qquad 
	dv_t =- \pa_x V(x_t) dt  \,,
	$$
	plus an Ornstein-Uhlenbeck (OU) process
	$$
	dv_t = -v_t dt+  \sqrt{2} dW_t \,.
	$$
	The generator of the Hamiltonian  dynamics is the 
	Liouville operator $B:= v\pa_x - \pa_x V(x)\pa_v$; this operator is antisymmetric in $\ltm (\R^2):=\{f:\R^2 \rightarrow \R : \int_{\R^2}|f(x,v)|^2 \mu(x,v)dxdv < \infty\}$. Because of this antisymmetry, along the  flow generated by $B$,  the $L^2_{\mu}$-norm  is conserved (see Subsection \ref{subsec: hypocoerc}). 
	The generator of the OU process is instead the operator $\cL_{OU} = -v \pa_v + \pa^2_v$; 
	by setting $A=\pa_v$, we have $\cL_{OU} = -A^*A$, 	where $^*$ denotes  adjoint  in $\ltm$, so that   $A^*= -\pa_v+v$.  Therefore, overall,  equation \eqref{kolmexample} can be written in the by now classic (linear) hypocoercive form 
	\be\label{FPhypocoercive}
	\pa_tu_t= Bu_t - A^*Au_t \, . 
	\ee
	While  $B$ is antisymmetric in $\ltm$, trivially,  the operator $A^*A$  is  symmetric  in $\ltm$; moreover, along the flow generated by $-A^*A$, the $\ltm$- norm is dissipated,  see again Subsection \ref{subsec: hypocoerc}.  Writing the dynamics in the above form,  is the starting point of the HT and it serves the purpose of emphasizing the splitting of the dynamics into its symmetric/dissipative  and antisymmetric/conservative parts. Once the dynamics has been cast in the form \eqref{FPhypocoercive}, dissipation of the $\ltm$-norm along the flow generated by $\cL$ becomes trivial to show, at least if one is not after a rate. The aim of the HT is to determine sufficient conditions on $A$ and $B$ such that dissipation towards the steady state $\mu$ is exponentially fast. 
	
	The setting of GENERIC is, in principle, analogous, however the theory is more adapted to working with the Fokker-Planck formulation of the dynamics; informally (precise definition in Subsection \ref{subsec:GENERIC}), we say that the evolution   \eqref{FPLangevin} is  in  GENERIC form \footnote{To be precise,  \eqref{FPgeneric} is  the pre-GENERIC form of the equation; pre-GENERIC is a more general formulation than GENERIC, we will recall the difference between such formulations in Section \ref{Sec2}. } if it can be written in the form  
	\be\label{FPgeneric}
	\pa_t \rho_t = W\rho_t - M_{\rho_t} \Big(\frac{1}{2}\vd\S(\rho_t)\Big) \,, 
	\ee
	where $W$ is some operator, $\S$ is a real-valued  functional, $\vd \S$ is the functional derivative of $\S$ (i.e. an appropriate derivative of $\S$ with respect to its argument $\rho$, see Section \ref{Sec2} ) and, for every fixed $\rho$,  $M_{\rho}$ is a symmetric positive semidefinite operator (in $L^2$); we furthermore require that the following orthogonality condition should hold:
	\be\label{ortintro}
	(W(\rho), \vd \S(\rho)) =0, 
	\ee
	where $( \cdot, \cdot)$ denotes scalar product in $L^2$. The FP equation \eqref{FPLangevin} can be cast in GENERIC form, 
	upon choosing $W=-B$, i.e. taking $W$ to be the opposite of the Liouville operator,  and furthermore choosing $\S$ to be the  {\em relative entropy} of the system, 
	\begin{align}
		\S_{\mu}(\rho)& = \int \rho(x,v)\log \left(\frac{\rho(x,v)}{\mu(x,v)} \right) dx \, dv \notag\\ 
		& \stackrel{\eqref{target}}{=} \int  \left(\rho(x,v) \log \rho(x,v)  +V(x)\rho(x,v)+\frac{v^2}{2} \rho(x,v)\right) dx \, dv \,, \label{eq: entropykFPE}
	\end{align}
	so that $\vd \S_{\mu}(\rho)$ denotes the variational derivative (Frechet derivative) of $\S$ with respect to $\rho$ (and then calculated in $\rho$), namely
	$$
	\vd\S_{\mu}(\rho) = \log \rho +1 + V(x)+ \frac{v^2}{2} \,;
	$$
	and, finally, by taking for each $\rho$, the operator $M_{\rho}$ (acting on a function $f$) to be defined as
	$$
	M_{\rho}(f):=-2 \pa_v\left[ \rho \,  \pa_v f\right] \,. 
	$$
	It is easy to show that, for each $\rho$ fixed,  $M_{\rho}$ is a symmetric operator in $L^2$, see Section \ref{subsec:GENERIC}.  By observing that the Liouville operator $B$ is antisymmetric in $L^2$ as well, one can see that the decomposition \eqref{FPgeneric} is in spirit analogous to the decomposition \eqref{FPhypocoercive}. The orthogonality condition ensures that the relative entropy $\S_{\mu}$ is dissipated along the flow.

	In view of the discussion that will follow it is important to point out that 
	the second addend of equation \eqref{FPgeneric}, i.e. the part of the equation that can be cast in the form $M_{\rho} (\vd\S(\rho))$, is the so-called {\em gradient-flow} part of the equation, see \cite{Dirr, kpel, kpelmath, MielkePeletierRenger2014}, more details on this in Section \ref{Sec2}. Hence the decomposition \eqref{FPgeneric} can also be viewed as a splitting into a  symmetric/dissipative/gradient-flow part and a antisymmetric/conservative/non-gradient flow component. As a bi-product of this discussion, we have the following: the generator $\cL_{OU}$ of the OU process is symmetric in $\ltm$ and associated to the dissipative/reversible part of the dynamics; its dual, $\cL_{OU}'$, which , using \eqref{FPLangevin},  coincides with $- M_{\rho}(\vd \S (\rho))$, is symmetric in $L^2$ and is associated to the gradient flow part of the dynamics.

	This is no coincidence and it is indeed a specific instance of a well-known `meta-theorem', substantiated by the work of Onsager and Machlup \cite{Onsager} as well as by e.g. \cite{Bertini, Kipnis1, Kipnis2, PeletierV}, according to which reversibility and gradient flow structure are related. One  way of expressing and understanding this meta-theorem is as follows: consider a stochastic particle system made (for simplicity) of i.i.d. particles,   suppose the process describing each  particle is reversible and that the particle  system converges to a deterministic limit, in the sense that the empirical measure associated with the particle system converges to a (deterministic) measure which satisfies a deterministic evolution equation.  Then such a deterministic equation is in gradient flow form. This meta-theorem, expressed in this form,  was settled into an actual theorem in \cite{MielkePeletierRenger2014}, for a large class of reversible processes. This was done by going through a Large Deviation Principle (LDP)  for the particle system. The reason why LDPs enter the picture is quite profound: it turns out that the form of the large deviation functional (LDF) appearing in the LDP is directly related to the gradient flow structure and indeed that it determines it. More specifically, if we know the form of the LDF then we can find $M$ (and $\cS$), see Section \ref{sec:4} and \cite{MielkePeletierRenger2014} for precise statements.   
	Note that typically every gradient flow evolution can be written in gradient flow form in more than one way, i.e. one can find several pairs $M$ and $\cS$ to write the same equation in gradient flow form; so if the objective is to  link reversibility and gradient flow structure,  the LDP helps `select' one such  gradient flow structure. This fact, i.e. the fact that the LDP `selects' one gradient flow form, is not in contradiction with the fact that the deterministic equation can be written in gradient flow form in more than one way: given a deterministic evolution, this evolution can be realised as limit of various  particle systems (e.g. the particle system need not be i.i.d.). Two different particle systems that converge to the same deterministic limit will, in general,  undergo different fluctuations and hence, assuming they  both satisfy a LDP, such LDPs will correspond to different large deviation functionals, which are then associated with different gradient flow structures \cite{PeletierV}.  This connection between evolutions with gradient flow structure and reversibility was then extended to a connection between evolutions in GENERIC form and `non-reversibility' in \cite{MielkePeletierRenger2014}, again making use of LDPs. In Note \ref{Note on decomposition} we elaborate on pros and cons of this approach. 
	
In view of our toy example it should be possible to state and prove the relation between reversibility and gradient flow structure without making use of LDPs. In this paper we state the mentioned `meta-theorem' as follows: assuming a given Markov process has a unique invariant measure $\mu$,  if the generator of the process is symmetric in $\ltm$ (which is equivalent to the process being reversible, see Section \ref{subsec:reversibility}) then the associated FP operator is in gradient-flow form. This allows us to prove this statement without making explicit use of LDPs. 
	
	{\bf Main Results. } We are now in a position to start explaining   the main contributions of this paper, section by section. After introducing setting and notation in Section \ref{Sec2}, in Section \ref{sec:3} we compare  HT and GENERIC -- more precisely, we compare linear HT and so called Wasserstein  pre-GENERIC, see Section \ref{Sec2} and Section \ref{sec:3} for precise statements and definitions. We show that if the KE associated with a given Markov process can be cast in the hypocoercive form \eqref{FPhypocoercive} then the corresponding FP equation can be written in  the form \eqref{FPgeneric}, {\em and} the orthogonality condition \eqref{ortintro} holds as well. We emphasize that the orthogonality condition for the FP evolution is implied by the hypocoercive structure of the KE, without extra assumptions, see Note \ref{note:linghyptopregen} on this point.   In particular, we provide explicit formulas to pass from one formulation to the other (i.e. given $A$ and $B$, we give formulas to obtain $W,M$ and $\cS$, and viceversa). On a practical level, having such ``conversion formulas" at hand allows one  to leverage results on the KE and potentially use them to produce results (almost for free) on the FP evolution, and viceversa. Since equations in GENERIC form dissipate relative entropy, a simple example of how one can produce results  by exploiting such connection is the following:  if the KE is in hypocoercive form \eqref{FPhypocoercive} then relative entropy is dissipated along the corresponding FP equation; see again Note \ref{note:linghyptopregen}. 
	
	More importantly from a conceptual standpoint, the work in Section \ref{sec:3}  provides a simple connection between ``non-reversibility'' and GENERIC structure; hence, by comparing  the symmetric part of both formulations,  it provides a  simple proof of the fact that reversibility gives rise to gradient flow structure, i.e. of the meta-theorem we mentioned above,   see Note \ref{note:linghyptopregen}.

	The results of Section \ref{sec:3} only refer to Markov processes generated by {\em linear}  \footnote{Here by {\em linear} we mean linear in the sense of McKean; i.e. the coefficients of the operator $\cL$ can be nonlinear in the state-space variable, as long as the action of the operator on functions is linear.}  Markov operators, which are typically simple to cast in the form \eqref{FPhypocoercive}. As we have already pointed out, the form \eqref{FPhypocoercive} is simply a symmetric-antisymmetric splitting (in $\ltm$), where the symmetric part is required to have a certain form, i.e. the form $-A^*A$. In Section \ref{section:decomposition}, we relax this assumption on the generator of the dynamics and produce analogous results to those in Section \ref{sec:3} , this time simply splitting the generator into its symmetric and antisymmetric part (again, symmetric and antisymmetric in $\ltm$) and assume that the symmetric part is linear, while the antisymmetric part is allowed to be non-linear. 
	While the main result of Section \ref{section:decomposition} is inspired by techniques used in large deviation theory, and in particular the approach of the proof is quite different from the one used in Section \ref{sec:3}, again no explicit use of LDP is made. 
	On this matter of linear vs non-linear generators, we point out that GENERIC, in the way we present it in Section \ref{Sec2} (which is a slight modification of the presentation in \cite{kpelmath}), is perfectly well adapted to include both linear and nonlinear dynamics.  The hypocoercive form \eqref{FPhypocoercive} is instead  typically only used for linear evolutions. Extensions of the HT to non-linear generators certainly exist, see e.g. \cite[Part III]{Villani}.  For purposes different from those of this paper it would be important to compare such non-linear formulations of HT with GENERIC; this will be subject of future work.
	 For the time being we point out that, remarkably, the same orthogonality condition \eqref{ortintro} which plays an important role in GENERIC, is also explicitly used in Villani's memoir (see \cite[Remark 38]{Villani}) to simplify verification of some of the (many) assumptions needed there.

	Coming back to LDPs,  in view of the fact that in the reversible case the gradient flow structure prescribes the form of the LDF (and viceversa), it is natural to ask whether a similar connection can be established in the non-reversible case. For this reason, in Section \ref{sec:4}, we consider generalised-reversible processes, which are a class of non-reversible processes, and we characterize the form of the LDF of generalised reversible processes. The only other result in this spirit which we know of is contained in \cite{kpelmath}, see also \cite{RZ2021,KJZ2018,Bouchet} {and a recent preprint \cite{Patterson2021} for similar works in this direction  for some other non-reversible processes (but these papers employ a different decomposition, namely decomposing fluxes/forces instead of the generator, and do not make connections to the HT)}.
	
	Finally, in Section \ref{sec:sec6} we show how the results of this paper can be applied to diffusion processes, Section \ref{Diffusion processes}, and to a class of PDMPs, Section \ref{sec:PDMPs}, the so-called Hamiltonian Piecewise Deterministic Markov Chain Monte Carlo processes (Ham-PD-MCMC). We consider diffusion processes for purely expository purposes, to show in a simple setting how to use the results of this paper,  but all the results of Section 	\ref{Diffusion processes} are known already, though perhaps not in the perspective in which we present them here. The PDMPs considered in Section \ref{sec:PDMPs} are considerably less standard; while the HT for Ham-PD-MCMC  has already been developed e.g. in \cite{Andrieu}, to the best of our knowledge GENERIC has not been applied to such processes yet. So, we first cast Ham-PD-MCMC in pre-GENERIC form and then we show that such dynamics constitute an example for which the orthogonality condition \eqref{ortintro} does not hold, and yet one is able to show entropy decay.

	To summarize, the paper is organised as follows: in Section \ref{Sec2}  we first clarify our setting, notation and standing assumptions (Subsection \ref{subsec:setup}); we then give a concise exposition of the HT, gradient flows and GENERIC, in Subsection \ref{subsec: hypocoerc} and Subsection \ref{subsec:GENERIC}.  We also point out that, while GENERIC is usually formulated on manifolds, we rephrase it here in a function space setting, mostly for ease of comparison with HT. {Besides this small modification, the content of Section \ref{Sec2} is well-known.}  In Section \ref{sec:3} we investigate the relation between hypocoercive and GENERIC formulation of the dynamics in the case in which the generator $\cL$ is linear. {The content of this section and of all subsequent sections (with the exception of Section \ref{Diffusion processes}) is new, to the best of our knowledge.} In Section \ref{section:decomposition}  we relax the class of generators for which it is possible to establish this connection and consider a class of non-linear operators. In Section \ref{sec:4} we characterize the structure of Large Deviation functionals of generalised-reversible Markov Processes. Finally, in Section \ref{sec:sec6} we apply the results of this paper to two classes of Markov processes, namely diffusion processes and PDMPs. 
	Overall, because this paper makes use of the work of two very different communities, we have tried to make it as self-contained as possible.

	\section{Background: Hypocoercivity and  Generic}\label{Sec2}
	
	In this section we recall the main facts about the  Hypocoercivity theory and GENERIC.  
	\subsection{Setup, Notation and preliminaries}\label{subsec:setup}
	Both the HT and GENERIC address the study of (stochastic) dynamics, more precisely of their associated Kolmogorov and/or Fokker-Planck PDEs.  Unless otherwise stated, all the stochastic processes $\{X_t\}_{\{t \geq 0\}}$  considered in this paper will be time-homogeneous Markov evolutions with finite dimensional state space; to fix ideas we will assume that the state space is $\R^d$.
	We recall that the semigroup $P_t$ associated with the Markov process $\{X_t\}_{\{t \geq 0\}}$ is defined on the set of functions $f:\R^d \rightarrow \R$ which are continuous and bounded by $(P_tf)(x):=\mathbb E [f(X_t) \vert X_0=x]$. Given $f$,  $(P_tf)(x)$ is a real- valued function of $(t,x)\in \R_+\times \R^d$; by Ito formula, such a function 
	solves a differential equation of the form 
	\be\label{kolmgen}
	\pa_t u(t,x) = \cL u (t,x), \quad u(0,x)=f(x),
	\ee
	where $\cL$ is an appropriate operator (e.g. a second order differential operator in the case in which $X_t$ is a diffusion process), called the Kolmogorov operator of the process $X_t$; correspondingly,  the differential equation \eqref{kolmgen} is called the {\em Kolmogorov equation} (KE) associated with $X_t$. 
	The dual equation, i.e. the equation
	\be\label{FPeqnsec2}
	\pa_t \rho(t,x) = \cL'\rho(t,x), \quad \rho(0,x) = \rho_0(x), 
	\ee
	where $\cL'$ is the (formal) $L^2$-adjoint of $\cL$, is referred to as the {\em Fokker-Planck equation} (FP) associated with $X_t$ and it describes the evolution of the law of $X_t$. A measure  $\mu$ on $\R^d$ is {\em invariant for $P_t$} (or, equivalently, invariant for $X_t$) iff
	\be\label{invmeasPt}
	\int_{\R^d} (P_t f)(x) \mu(x) dx = \int_{\R^d} f(x) \mu(x) dx \quad \mbox{for every } f \in C_b(\R^d) \,.
	\ee

	There are various ways of specifying the domain of the (typically unbounded) operator $\cL$.  If there exists an invariant measure then the semigroup extends to a strongly continuous semigroup on $\ltm$;  in this case  $\cL$ is the generator of such a semigroup and we take as a domain for $\cL$ the domain $\cD_2({\cL})$ of $\cL$ viewed as generator of $\cP_t$ in $\ltm$, see \cite[Definition 1.5]{GuionnetZegarlinski}.  
	 \footnote{We clarify that  we are abusing notation by denoting by $\cL$ both the Kolmogorov operator and the generator of the semigroup, as  these two objects only formally coincide; indeed the former is just a formal expression given by the It\^o formula and most complications in this context arise precisely from trying to reconcile the two objects, i.e. from finding appropriate domains of definition see \cite[Chapter 1 and Chapter 3]{bakry2013analysis} for details.}   One can show that  \eqref{invmeasPt} is  equivalent to (see \cite[Section 1.2]{GuionnetZegarlinski})  $\int({\cL} f)(x) \mu(x) dx =0$ for every $f \in \cD_2({\cL})$. In short, this can be written as  
	\be\label{invmeasLprime}
	\cL'\mu  =0\,. 
	\ee

	For simplicity, we make the following {\bf standing assumptions}, which will hold throughout the paper,  unless otherwise specified:
	\begin{hypothesis}\label{SA} Standing assumptions:
		\begin{enumerate}[label=\textbf{\textup{[SA.\arabic*]}}]
			\item All the probability measures  which we consider have a density with respect to the Lebesgue measure  and, with abuse of notation, we denote the measure and its density with the same letter, i.e. $\mu (dx) = \mu(x) dx$.\label{SA2}
			\item  We only consider {\em ergodic} Markov processes, i.e. the underlying process $X_t$ admits a unique invariant measure, $\mu$, which has a (smooth enough) density with respect to the Lebesgue measure.  \label{SA1} 
			\item We assume that $\cL'$ and the initial datum $\rho_0$  are such that the FP equation \eqref{FPeqnsec2} admits strictly positive classical solutions, which additionally belong to $L^2(\R^d) \cap L^1(\R^d)$.   \footnote{Note that typically Fokker-Planck type equations preserve positivity, see \cite{BogKrylRoeckner}}   \label{SA3}
			\item The operator $\cL$ is the generator of a strongly continuous contraction semigroup on a Hilbert space $\cH$ (which will typically be $\ltm$). 
		\end{enumerate}
	\end{hypothesis}
	
	Let us motivate  the above assumptions and point out explicitly what they imply. 
	
	$\bullet $ Under Assumptions \ref{SA2} and \ref{SA1}, equation \eqref{invmeasLprime} has a unique classical solution;   without loss of generality, we can then assume $\mu(x)>0$ for every $x\in\R^d$ and we shall do so throughout. This will allow us to consider the equation solved by the function $h(x,t) = \rho(x,t) \mu(x)^{-1}$, which we call the {\em modified Kolmogorov equation} (m-KE)---we will be more precise on the relation between KE, FP equation and m-KE in Note \ref{lem: KE vs FPE} below. 
	Let us also recall that if $\mu$ is the unique invariant measure for the semigroup then it is an ergodic measure.  If $\cL$ is linear and  the semigroup is stochastically continuous (\cite[Section 2.1]{daPrato}) this implies that the kernel of the operator $\cL$ in $\ltm$ is made only of functions which are $\mu$-a.s. constant.\footnote  {The fact that $\cL {\bf 1} =0$ follows simply by the Hille-Yoside Theorem. Viceversa, the ergodicity of the measure is equivalent (see \cite[Theorem 3.2.4]{daPrato}) to either one of the following statements: 
		i) if $f \in L^2_{\mu}$ then
		$
		P_tf =f \, \Rightarrow \, f={\mbox{constant}}, \,  \mu-a.s.$;
		ii) if $f \in \cD_2(\cL)$, then
		$
		\cL f =0 \, \Rightarrow \, f={\mbox{constant}}, \,  \mu-a.s.$. 
		To see that statement i) and ii) are equivalent use \cite[equation (1.1.6)]{GuionnetZegarlinski} and recall that the domain of the generator is dense in $\ltm$.}
	
	$\bullet $  In general, the solution of the m-KE and of the FP equation don't lie in the same space; it is usually more natural to study the KE or the m-KE  in the weighted space $\ltm:=\ltm(\R^d;\R):=\{f:\R^d \rightarrow \R \mbox{ s.t. } \int f^2(x) \mu(dx)< \infty\} $,   and the FP equation in $L^2:=L^2(\R^d;\R)$ (or in $L^2$ weighted with an appropriate polynomial or, better yet, simply in measure-space,  see for example \cite[Section 2.4]{Villani}). For this reason we introduce two Hilbert spaces of real-valued functions, $(\cH, \|\cdot\|, \lan \cdot, \cdot \ran)$ {and $(\cZ, ((\cdot)), (\cdot, \cdot ) )$, to be thought of as  $\ltm$ and $L^2$, respectively. In particular, unless otherwise specified, $(\cdot, \cdot)$ will denote $L^2$ scalar product (or duality pairing, we will specify which in context, when needed).} For the KE and the m-KE we don't necessarily restrict to considering classical solutions and, since we assume that $\cL$ generates a strongly continuous contraction semigroup in $\ltm$, the semigroup itself is the solution in $\ltm$ of the KE. 
	
	$\bullet $	Technical matters about Kolmogorov and Fokker-Planck equations are often quite  involved (see e.g. \cite[Chapter 3]{bakry2013analysis}) and are better treated on a case by case basis; GENERIC is unsurprisingly plagued by the same problems  (see e.g. \cite[Remark 1.1]{DuongPeletierZimmer2013} and references therein).   In this paper  we don't restrict to a particular class of Markov processes (e.g. we don't restrict to diffusion processes); on the contrary, we would like to include a broad range of Markov dynamics (see examples in Section \ref{sec:sec6}) and at the same time we want to avoid encumbering the exposition with  excessive technicalities.
	For this reason, we assume the existence of an appropriate set $\dense$ which, unless otherwise stated,  is dense both  in $\cD_2(\cL) \cap \cD_2(\cL^*)$  and in the space $\csol$:=$\{$ positive classical solutions of the FP equation which live in $L^2 \cap L^1$ $\}$.  In practical examples $\dense$  can often be taken to be the set of (positive) Schwartz functions;  this way we can perform all our calculations on $\dense$.   In this sense -- i.e. in the sense that we only work on $\dense$ --  some calculations in this section and the next will be somehow formal. We flag up now and we will come back to this (see Note \ref{note:how to recover the metric}) that even doing this does not completely solve the problem for GENERIC and some remarks will be necessary.  In the examples of Section \ref{sec:sec6} we will give indications on the correct functional framework.

	For any operator or functional, say $\mathcal T$, $\cD(\mathcal T)$ denotes a domain of definition of $\mathcal T$ \footnote{When we want to refer to a specific domain we will do so, see e.g. the difference between $\cD_2$ introduced earlier in this section and $\cD_b$, used in Section \ref{sec:4}. Otherwise $\cD$ generically denotes a set on which the operator is well defined. } and we will use interchangeably $\mathcal T h$ and $\mathcal T(h)$ to denote the action of $\mathcal T$ on an element $h \in \cD(\mathcal T)$. Moreover, for any function $f(t,x)$ depending on both time and space, the notations $f_t(x)$ and $f(t,x)$ will be used interchangeably.  With this  in mind we clarify the relation between the FP equation, the KE and the m-KE.
	
	\begin{Note}[relation between KE, m-KE and FP equation]
		\label{lem: KE vs FPE}
		\textup{
			Suppose that the Kolmogorov operator $\cL$ in \eqref{kolmgen} is of the form 
			$$\cL = B+ \tilde{A}, $$ 
			where $B$ and $\tilde{A}$ are linear operators  with $B$ antisymmmetric in $\ltm(\R^d)$ and $\tilde{A}$ symmetric in $\ltm(\R^d)$. We also assume that $B$ enjoys the chain and product rule.  Let $\rho_t$ be the solution of the FP equation associated with the process,
			$
			\pa_t \rho_t = \cL' \rho_t \,, 
			$
			and set  $\rho_t = h_t \mu$. Then the function $h_t$ solves the m-KE, namely
			\be\label{m-KE}
			\pa_t h_t = \mko  h_t, \quad \mbox{where } \,\,\mko = (-B + \tilde{A}) \,.
			\ee}
		\textup{To see the above it suffices to show what follows:
			\begin{equation}
				\label{eq: KE vs FPE}
				\cL'(h_t \mu)= \mu \,  (\mko h_t).
			\end{equation}
			Indeed if \eqref{eq: KE vs FPE} holds, we then have 
			$$
			\pa_t \rho_t = \mu \pa_t h_t =	\cL'\rho_t=\cL'(h_t \mu) = \mu \,  \mko h_t \,,
			$$
			which directly implies \eqref{m-KE}. 
			To show \eqref{eq: KE vs FPE}, notice that for every $f,h \in \dense$, we have 
			\begin{align*}
				& \int B'(h\mu)(x)f(x)\,dx=\int h(x)\mu(x) (B f)(x)\,dx=-\int f(x)(B h)(x)\mu(x)\,dx,
			\end{align*}
			and, similarly, 
			\begin{align*}
				&\int \tilde{A}'(h\mu)(x)f(x)\,dx=\int h(x)\mu(x) (\tilde{A} f)(x)\,dx=\int f(x)(\tilde{A} h)(x)\mu(x)\,dx,
			\end{align*}
			which imply  
			\be\label{xx}
			B'(h\mu)=-\mu B h  \quad \mbox{ and }  \quad  \tilde{A}'(h\mu)=\mu\tilde{A} h \, ,
			\ee
			from which \eqref{eq: KE vs FPE} follows, as $\cL'=B'+\tilde{A}'$. 
			Note that, from \eqref{eq: KE vs FPE} one can formally see that in this setting $\mu$ is a solution of $\cL' \mu =0$ if and only if the kernel of $\mko$ contains constants. In short, 
			\be\label{invmeasiffL1=0}
			\cL' \mu = 0 \Leftrightarrow \mko {\bf 1} =0,
			\ee
			where ${\bf 1}$ denotes the function identically equal to one.} 
	\end{Note}
	\smallskip
	
	We recall  that if  $\cS$ is a functional on the Hilbert space  ($\cZ$ $( \cdot, \cdot ), (( \cdot ))$), $\mathcal S:\cD(\cS)\subseteq\cZ \rightarrow \R$, the (directional) Gateaux derivative of $\cS$ at $h$ in the direction $g\in \cZ$ is given by
	\be\label{varderdef}
	( \vd\S (h) , g ) = \frac{d}{d \epsilon} \S (h+ \epsilon g) \Big \vert_{\epsilon = 0} \,.
	\ee
	The Frechet derivative of $\S$ at $h$ is the linear functional $\vd \S(h)$ mapping $g$ into the directional derivative of $\S$ at $h$ in the direction $g$. When $\cS$ is defined on the whole Hilbert space $\cZ$ then, for each fixed $h$, one can view $\vd \cS(h)$ as an element of $\cZ$ and on the LHS of \eqref{varderdef} the notation $(\cdot, \cdot)$ is an actual scalar product. 
	This is easier to see in finite dimensions. If $\cZ$ is finite dimensional, i.e. if $\cZ = \R^d$ and $f:\R^d \rightarrow \R$, then the Frechet derivative of $f$ coincides with the differential and it will be denoted, as customary, by $\nabla f$. In this case, for each $x \in \R^d$, $\nabla f(x) \in \R^d$ while the map $\nabla f : \R^d \rightarrow \R^d$ may be viewed as an element of $\cZ^d:= \{u=(u^1 \dd u^d): u^j \in \cZ \mbox{ for all } j\}$.
	However in practical examples, for our choice of the infinite dimensional space $\cZ$ (typically $\cZ=L^2$) it is almost never the case that $\cS$ is defined on the whole of  $\cZ$  and, for fixed $h \in \cZ$, $\vd \cS$ is not in $\cZ$ but in a bigger space. In this case the scalar product of $L^2$ ($\cZ$) will only act as a formal duality pairing, see Subsection \ref{subsec:GENERIC}. 
	
	We will use interchangeably the notation $\nabla \cdot$ and $\div$ for the divergence of a $\R^d$-valued function. 
	The Euclidean norm of $\R^d$ will be denoted by $\vert\cdot \vert$.


	\subsection{Hypocoercivity}\label{subsec: hypocoerc}
	In this section we  fix   $\cH:= \ltm(\R^d)$, where $\mu$ is assumed to be the only invariant measure of the process with generator $\cL$, FP operator $\cL'$ and modified Kolmogorow operator $\mko$. We say that a Markov evolution is  in {\em linear hypocoercive form } if its generator $\cL$ is linear and can be written in the form 
	\be\label{m-KElinearhypocform}
	\cL = B  - A^*A, 
	\ee
	where  $A: \dense \rightarrow (\ltm)^d $ and $B: \dense \rightarrow \ltm $ are  linear  operators, $A$ obeys the chain rule,    $B$ is antisymmetric in $\ltm$, i.e. $B^*=-B$, and obeys both chain and product rule.  
	We recall that under our standing assumptions $\ker \cL = \{constants\}$. Because $A^*$ denotes adjoint in $\ltm$, $A^*A$ is necessarily symmetric in $\ltm$.  \footnote{ We clarify that $A$ should be thought of as  a $d$ - dimensional vector of operators $A= (A_1 \dd A_d)$, e.g. $A=\nabla = (\pa_{x_1}, \dots, \pa_{x_d})$. We also adopt the same understanding as in \cite{Villani}, namely $A^*A$ is a short notation for $A^*A= \sum_1^d A_i^*A_i$.}
	
	If $\cL$ is of the form \eqref{m-KElinearhypocform} the m-KE \eqref{m-KE} takes the form \be\label{m-KEactual}
	\pa_t h_t = - (B+ A^*A) h_t \,.
	\ee
	Let us look at the two parts of the dynamics: along the flow generated by $-B$, i.e. along the solution of 
	the equation $\pa_t f_t = {-}B f_t$, the norm is conserved; indeed, since $B$ is antisymmetric in $\ltm$, $\langle Bf, f\rangle =0$ for every $f \in \dense$, hence
	$$
		\pa_t \| f_t\|^2 = -2 \langle B f_t, f_t\rangle =0. 
	$$
	On the other hand, along the flow generated by $-A^*A$, the norm is dissipated:
	\be\label{pre20}
	\pa_t \| f_t\|^2 = -2 \langle  f_t, A^*Af_t\rangle = -2 \| A f_t\|^2 \leq 0\,. 
	\ee
Overall, with similar calculations,  the $\ltm$- norm of $h_t$ is dissipated:
	\be\label{dissipationofnorm}
	\pa_t \| h_t\|^2 = -2 \| A h_t\|^2 \leq 0\,.
	\ee
	For later comparisons with GENERIC we observe that this formalism does not imply the existence of any conserved quantity. Note that if the generator is in hypocoercive form then 
	\be\label{kerl-kerB int KerA}
	\ker \cL = \ker \mko = \ker A \cap \ker B\,,
	\ee
	see \cite[Proposition 2]{Villani}.
	
	We gather in the next lemma some straightforward facts which will be useful later on.
	\begin{lemma}\label{lemma:B*=B'=-B}
		With the notation introduced so far, let   $W: \dense \rightarrow \ltm$ be a linear operator such that $W$  and $W'$    satisfy the chain and product rule (e.g., a first order differential operator) and both $W'$ and $W^*$ are well-defined on $\dense$.  
		\begin{itemize}
			\item[i)] If  $W'\mu =0$, \footnote{Here we are implicitly assuming that $\mu$ is smooth enough that the operator $W'$ can be applied to $\mu$.}  then 
			$$
			W'h=W^*h= -Wh, \quad \mbox{for every } h \in \dense,  
			$$
			i.e. $W$ is antisymmetric both in $L^2$ and in $\ltm$;
			\item[ii)] If the generator is in hypocoercive form  \eqref{m-KElinearhypocform} and $\ker \cL = \{constants\}$ then $B'\mu =0$. By using point i), if both $B$ and $B'$ satisfy chain and product rule,  this implies $B'=B^*=-B$ and hence also $B \mu =0$. 
		\end{itemize} 
	\end{lemma}
	\begin{proof}[Proof of Lemma \ref{lemma:B*=B'=-B}]To prove i),  we first prove that $W'h=W^*h$ for every $h \in \dense$; trivially, if $W'\mu =0$ then for every $f,g \in \dense$
		\begin{align*}
			\int_{\R^d} (Wf)(x) g(x) \mu(x) dx & = \int_{\R^d} f(x) (W^*g)(x) \mu(x) dx \\
			& = 
			\int_{\R^d} f(x) W'(g \mu)(x) dx = \int_{\R^d} f(x) (W'g)(x) \mu(x) dx \,.  
		\end{align*}
		We now show $W^*=-W$ (on $\dense$) by proving that $\lan W g, g\ran = \int (Wg)(x) g(x) \mu(x) dx =0$ for every $g \in \dense$. Indeed:   
		$$\lan W g, g\ran = \frac{1}{2} \int_{\R^d} [W(g^2)](x) \mu(x) dx = \frac{1}{2} \int_{\R^d} g^2(x) (W' \mu)(x) dx =0 \,, $$ 
		which  concludes the proof of i). As for ii),  if $\cL {\bf 1}=0$, by \eqref{kerl-kerB int KerA}, $B{\bf 1}=0$ as well, hence
		$$
		0= \int_{\R^d} (B{\bf 1}) \mu(x) g(x) dx =- \int_{\R^d} \mu(x) (Bg)(x) dx = {-}\int_{\R^d} (B'\mu)(x) g(x)  dx \,,\quad \mbox{for every } g \in \dense , 
		$$
		where the first equality comes from using $B^*=-B$.
	\end{proof}
	An example which notoriously satisfies the assumptions of the above theorem is the Liouville operator $B$ discussed in the introduction,  which is indeed antisymmetric both in $L^2$ and in $\ltm$, with $\mu$ the Boltzmann distribution.  
	\begin{Note}\label{note:aim}
		The aim of the hypocoercivity theory  \cite{Villani, Herau} is to study exponentially fast decay to equilibrium when the operator $\cL$ is not coercive in the $\ltm$ norm, but it is instead coercive in a modified norm.  In the HT writing the splitting \eqref{m-KElinearhypocform}   is only a starting point, and  for this reason  we refer to equations driven by operators in the form \eqref{m-KElinearhypocform} as being in linear hypocoercive form; however for the sake of clarity we emphasize that the fact that the operator $\cL$ is in the form \eqref{m-KElinearhypocform} does not mean per se  that it is hypocoercive (see \cite[Section 3]{Villani} for a definition of hypocoercive operator) and hence that exponentially fast decay holds,   as  further quantitative assumptions on $A$ and $B$ are needed to that end. In this paper we compare primarily the structure of the involved equations, so we don't make any such quantitative assumptions. Understanding the meaning of such assumptions in the context of GENERIC will be the object of future work. 
	\end{Note}

	\subsection{Gradient Flows, GENERIC and pre-GENERIC}\label{subsec:GENERIC}
	The classical formulation of  GENERIC  is on  manifolds \cite{kpelmath} (and this is perhaps a more insightful way of understanding the theory). Here, 
	to compare more easily with the HT, we start by reformulating GENERIC and pre-GENERIC in a function-space setting and we work on $\dense \subseteq \cZ$, where again $\cZ$ is to be thought of $L^2$, although the scalar product of $L^2$ will often only act as a formal duality pairing.   We make remarks on why this is the case and on  how this formulation compares with the classical one on manifolds throughout the paper,  see in particular Note \ref{note:how to recover the metric} and Note \ref{note:what Lagrangian be on manifolds}.

	$\bullet $ {\bf Quadratic Gradient flows}.  Let $\{\rho_t\}_{t \geq 0} \subset \dense$ and $\S: \dense \rightarrow \R$.  An evolution equation of the form
	\be\label{quadraticgradientflow}
	\pa_t \rho_t = - M_{\rho_t}\Big(\frac{1}{2}\vd \S (\rho_t)\Big)
	\ee
	is a {\em (quadratic) gradient flow} if
	for each $\rho \in \dense$ s.t. $\rho>0$, $M_\rho: \cD(M_{\rho}) \supseteq \dense \rightarrow \dense$ is a symmetric and positive semidefinite operator, i.e.
	\be\label{positiveM}
	( M_{\rho}(h), g ) = ( h, M_{\rho} (g) ), \quad \mbox{and} \quad 
	( M_{\rho}(g), g ) \geq 0 \quad \forall h,g \in \dense \,,
	\ee
	$\cS$ is Frechet differentiable and  $\vd \cS(\rho) \in \cD(M_{\rho})$ for every $\rho \in \dense$. 
	If
	$M_{\rho} (\cdot) = -2\div (\rho \nabla \cdot)$, then \eqref{quadraticgradientflow} is a {\em Wasserstein gradient flow}, see \cite{Dirr}. If we can write $M_{\rho}$ as 
	\be\label{genWasgradflow}
	M_{\rho} (\cdot) = 2 A'(\rho A(\cdot))\, 
	\ee
	for some operator $A$ and $L^2$-dual $A'$, then we say that \eqref{quadraticgradientflow} is a {\em generalised Wasserstein gradient flow}. Note that  the operator $M_{\rho}$ in \eqref{genWasgradflow} is symmetric and positive definite (we show this explicitly after \eqref{kerBorthogonality}); moreover,  when $A=\nabla$ we recover $M_{\rho} (\cdot) = -2\div (\rho \nabla \cdot)$.
	
	As a consequence of the gradient-flow structure, the {\em entropy} $\cS$ is dissipated along the solution of \eqref{quadraticgradientflow}:
	\be\label{dissSquadratic grad flow}
	\frac{d\cS (\rho_t)}{dt} = (\vd\cS(\rho_t), \pa_t \rho_t) =
	- \Big(\vd  \cS (\rho_t), M_{\rho_t}\big(\frac{1}{2}\vd \S (\rho_t)\big)\Big) \leq 0 \,,
	\ee
	having used the positivity of $M_{\rho}$. We emphasize the analogy between the above calculation and the one in \eqref{pre20}.  
	Finally, there is no natural conserved quantity associated with the flow \eqref{quadraticgradientflow}.

	\begin{Note}\label{note:how to recover the metric}   If $\cZ$ was a general manifold (as in \cite{MielkePeletierRenger2014, kpelmath}) instead of a Hilbert space and $\cS:\cZ \rightarrow \R$  then, by definition \eqref{varderdef} of variational derivative,  $\vd\cS(\rho) \in T^*_{\rho}\cZ$, where $T^*_{\rho}\cZ$ is the cotangent space of $\cZ$ at $\rho$,  and $M_{\rho}:T^*_{\rho}\cZ \rightarrow T_{\rho}\cZ$, where $T_{\rho}\cZ$ is the tangent space of $\cZ$ at $\rho$, for every $\rho \in \cZ$. (Note that any Hilbert space $\cZ$ is a Hilbert manifold and both  $T^*_{\rho}\cZ$ and $T_{\rho}\cZ$  are isomorphic to $\cZ$.)
		This approach is somehow neater and, as we said, possibly more insightful. The difficulty here becomes to first identify the correct manifold $\cZ$ and then describe explicitly $T\cZ$ and $T^*\cZ$. This is not straightforward when e.g. $\cZ=\{$ finite positive Borel probability measures$\}$, see \cite[Section 3]{MielkePeletierRenger2014}. Further issues then arise when introducing dissipation potentials, as we will do in the next subsection on gradient flows, as such functionals are then typically not defined on the whole tangent and cotangent space. To avoid these complications we simply work on a very good set, where everything is well posed. However also this approach needs some extra comments, as it is clear from the example we gave in the Introduction, where $(\vd\cS(\rho))(x,v) = \log \rho + V(x)+1+ v^2/2$, that the function $\vd\cS(\rho)$ does not belong to $L^2$ but for example it belongs to $\ltm$ instead. In that example all would work by choosing $\cD(M_{\rho})$ to be the set of smooth functions in $\ltm$. Hence in \eqref{positiveM} or in \eqref{dissSquadratic grad flow} the scalar product of $L^2$ acts only as a formal pairing, in the sense that, upon making   a good choice of $\cD(M_{\rho})$, all the integrations implied by the formal $L^2$ scalar product make sense and this is how we will understand them throughout (including in the expressions \eqref{legtrans} and \eqref{legtrans1} below). 
	\end{Note}

	$\bullet $ {\bf Gradient flows}. To introduce a more general (non-quadratic) notion of gradient flow,  let us first give the definition of {\em dissipation potential}. For every $\rho  \in \dense$ let $\psi(\rho;\cdot): \cD(\psi_{\rho}) \subseteq\dense \rightarrow \R$ and let $\psi^{\star}(\rho; \cdot): \cD(\psi^{\star}_{\rho}) \supseteq \dense \rightarrow \R$ denote its Legendre transform, namely
	\begin{align}\label{legtrans}
		\psi^{\star}(\rho;\xi) &= \sup_{v } 
		\{(v,\xi) - \psi(\rho;v)\}\, . 
	\end{align}
	We recall that, given a convex, \footnote{ We say that a function $F$ is convex if  $F(\lambda \xi_1+(1-\lambda)\xi_2)\leq \lambda F(\xi_1)+(1-\lambda)F(\xi_2)$ for all $\lambda\in[0,1]$} lower semicontinuous function $F$, not taking the value $-\infty$, we have  $F^{\star\star}=F$, see \cite[Proposition 3.1 and Proposition 4.1]{Ekeland1976} for details. Under such assumptions on $\psi$, we  have    
	\begin{align}\label{legtrans1}
		\psi(\rho; v) = \sup_{\xi } 
		\{(v,\xi) - \psi^{\star}(\rho;\xi)\}\, ,
	\end{align}
	i.e. $\psi^{\star \star}=\psi$ (at least on a good enough set of functions). 
	We will use interchangeably the notation $\psi(\rho;v)$ and $\psi_{\rho}(v)$ (same for $\psis$). 
	We say that the pair $\psi, \psi^{\star}$ is a pair of
	{\em dissipation potentials} if both $\psi$ and $\psi^{\star}$ are strictly convex,  continuously differentiable \footnote{This assumptions can be relaxed. If $\psi^\star$ is not differentiable, we can substitute its derivative by its convex subdifferential, see \cite[Section 2.5]{Mielke2011}} and 
	\be\label{psipsistarzeroatzero}
	\psi(\rho;0)= \psi^{\star}(\rho;0)=0, \quad \mbox{for all } \rho \in \dense \,.
	\ee
	The functions $\psi$ and $\psi^{\star}$ are symmetric if $\psi(\rho;v) = \psi(\rho;-v)$ and $\psi^{\star}(\rho;\xi)= \psi^{\star}(\rho;-\xi)$, for every $\rho, v, \xi$. We recall that, by the definition of  Legendre transform, i.e. by using \eqref{legtrans} and \eqref{legtrans1}, 
	\begin{align}
		& \psi(\rho;0)=0  \,\,\Leftrightarrow \,\, \inf\psi^{\star}(\rho;\cdot) = 0 \Rightarrow \psi^{\star}(\rho;\cdot) \geq 0, \label{pos1}\\
		& \psi^{\star}(\rho;0)=0  \,\,\Leftrightarrow \,\, \inf\psi(\rho;\cdot) = 0 \Rightarrow \psi(\rho;\cdot) \geq 0. \label{pos2}
	\end{align}

	A {\em gradient flow} with respect to a functional $\cS:\dense \rightarrow \R$ and dissipation potentials $\psi, \psi^{\star}$ is an evolution equation of the form 
	\be\label{non-quadratic gradient flow}
	\pa_t \rho_t = \vd_{\xi} \psi^{\star} \left(\rho_t; -\frac{1}{2} \vd\cS(\rho_t) \right) \,.
	\ee
	Equivalently, one can show (see \cite[Section 2.1]{kpelmath}) that an evolution equation $\pa_t \rho_t = \vecfield (\rho_t)$ (where $\vecfield: \cD(\vecfield) \supseteq \dense  \rightarrow \dense$ is an operator) is a gradient flow with respect to $\cS, \psi, \psi^{\star}$ iff
	\begin{equation}\label{gradflowequivdefinition}
		\psi(\rho_t; \vecfield(\rho_t)) + \psi^{\star}\left(\rho_t; -\frac{1}{2} \vd \cS(\rho_t)\right)+ \frac{1}{2} ( \vecfield (\rho_t), \vd \cS(\rho_t))=0 \,  .  
	\end{equation}
	To clarify the notation in the above, on the RHS of \eqref{non-quadratic gradient flow} $\vd_{\xi}$ is the Frechet derivative of $\psis(\rho;\xi)$ with respect to its second argument (then calculated at $-\vd\cS(\rho)/2$). As  we will always regard $\psis(\rho;\xi)$ as a function of $\xi$ for every $\rho$ fixed, we could have dropped the subscript $\xi$, which is there both for clarity and to keep closer to tradition.

	If $\cZ = \R^d$ (which, strictly speaking, is not allowed by our setup, as we assumed that $\cZ$ is a space of real-valued functions) and  $\psis_{\rho}(\xi)= M \xi\cdot  \xi $, where $M$ is a symmetric positive definite matrix,  then \eqref{non-quadratic gradient flow} boils down to \eqref{quadraticgradientflow}, hence the name {\em quadratic} gradient flow for \eqref{quadraticgradientflow}.

	A calculation analogous to \eqref{dissSquadratic grad flow} -- this time exploiting the convexity of $\psis$\footnote{We recall that if $F$ is convex and continuously differentiable then $(x-y) (\vd F(x) - \vd F(y)) \geq 0$ for every $x,y$. } rather than the positivity of $M$ --  allows again to show that the entropy $\cS$ is dissipated along the flow.

	$\bullet $ {\bf pre-GENERIC}. Let $W:\cD(W) \supseteq \dense  \rightarrow \dense$ be an operator and $\cS$ and $\psis$ as above. Then an evolution equation of the form 
	\be\label{pre-GENERIC}
	\pa_t\rho_t = W(\rho_t) + \vd_{\xi}\psis_{\rho_t}\left(-\frac{1}{2}\vd\cS(\rho_t) \right)
	\ee
	is said to be a {\em pre-GENERIC} flow (with respect to $W, \cS, \psi, \psi^{\star}$) if the following degeneracy condition is satisfied:
	\be\label{degconpre-GENERIC}
	(W(\rho), \vd \cS(\rho)) =0 \quad \mbox{for all } \rho \in \dense \,.
	\ee
	Equivalently (see Appendix \ref{app:equiv}) the evolution equation $\pa_t \rho_t = \vecfield (\rho_t)$ is a {\em pre-GENERIC} flow (with respect to $W, \cS, \psi, \psi^{\star}$) iff  \eqref{degconpre-GENERIC} holds and 
	\begin{equation}\label{equivdefpre-generic}
		\psi(\rho_t;\vecfield(\rho_t) - W(\rho_t)) + \psi^{\star}\left(\rho_t;-\frac{1}{2}  \vd \cS(\rho_t) \right) + \frac12(\vecfield(\rho_t), \, \vd \cS(\rho_t))=0 \,.
	\end{equation}
	\begin{Note}\label{note:variousondissip}
		Some comments on the above definition
		\begin{description} 
			\item[i)]If the gradient flow part of \eqref{pre-GENERIC} is quadratic, i.e. if the evolution is given by 
			\be\label{quadratic pre-generic}
			\pa_t \rho_t = W(\rho_t)- M_{\rho_t} \left(\frac{1}{2}\vd \cS(\rho_t)\right), 
			\ee
			still appended with the condition \eqref{degconpre-GENERIC}, then we talk about {\em quadratic pre-GENERIC}; if the flow is {quadratic pre-generic} with $M_{\rho}$ given by \eqref{genWasgradflow}, then we talk about {\em generalised Wasserstein pre-GENERIC}.
			\item[ii)]If \eqref{pre-GENERIC} and \eqref{degconpre-GENERIC} hold, then
			 $\cS$ decays along the flow. However, if the purpose is to have entropy dissipation, then \eqref{degconpre-GENERIC} can be replaced with the following dissipativity condition
			 \be\label{dissipcond}
			 (W(\rho), \vd \cS(\rho)) \leq 0 \quad \mbox{for all } \rho \in \dense \,.
			 \ee
			Also from the point of view of the equation structure, one can dispense with the orthogonality condition \eqref{degconpre-GENERIC} at the price of modifying \eqref{equivdefpre-generic}. Indeed one can see (see Appendix \ref{app:equiv})  that the flow $\pa_t \rho_t = \vecfield (\rho_t)$ is of the form \eqref{pre-GENERIC} if and only if
			\begin{equation}\label{equivdefpre-genericnoorthogonality}
				\psi(\rho_t;\vecfield(\rho_t) - W(\rho_t)) + \psi^{\star}\left(\rho_t;-\frac{1}{2}  \vd \cS(\rho_t) \right) + \frac12(\vecfield(\rho_t) - W(\rho_t), \, \vd \cS(\rho_t))=0 \,.
			\end{equation}
		However  entropy decay does not necessarily hold if \eqref{equivdefpre-genericnoorthogonality} holds in place of \eqref{equivdefpre-generic}-\eqref{degconpre-GENERIC}.  If \eqref{equivdefpre-genericnoorthogonality} is appended with the dissipativity condition   \eqref{dissipcond} 
			then the entropy functional does decrease along the flow. Indeed,  using the non-negativity of $\psi$ and $\psi^{\star}$ (see \eqref{pos1}-\eqref{pos2}), \eqref{equivdefpre-genericnoorthogonality} and \eqref{dissipcond}, one has
			\begin{align*}
				\frac{1}{2}\frac{d}{dt}\cS(\rho_t)&=\frac{1}{2}\left(\dot{\rho_t},\vd \cS(\rho_t)\right)=\frac{1}{2}\left(\vecfield(\rho_t),\vd\cS(\rho_t)\right)
				\\&=-\psi(\rho_t,\vecfield(\rho_t)-W(\rho_t))-\psi^*(\rho_t,-\frac{1}{2}\vd \cS(\rho_t))+\left( W(\rho_t), \frac{1}{2} \vd \cS(\rho_t)\right) \leq 0\,.
			\end{align*}
			Hence $\cS$ is a Lyapunov functional of the dynamics. 
			\item[iii)]Since $\psi^{\star}$ is assumed to be symmetric and $\psi$ is its Legendre dual, then also $\psi$ is symmetric. 
		\end{description}
	\end{Note}

	\subsection{The many meanings of  the word ``reversibility"}\label{subsec:reversibility}
	It is rather unfortunate that the word ``reversible" is used by different communities  with substantially opposite meanings. This misunderstanding becomes particularly confusing in the context of this paper, so we make some clarifications but refer the reader to \cite{Pavliotis} and \cite{Lelievre2010} for complete statements and proofs. A continuous-time stochastic process $\{X_t\}_{t\geq 0}$ is called {\em time-reversible} or simply {\em (microscopically) reversible} if its law is invariant under time-reversal, i.e. if for every fixed $T>0$ the process $\{X_t\}$ and the process $\{X_{T-t}\}$ have the same distributions on the space of continuous paths. Intuitively, this can be interpreted as follows: the process is time-reversible if by watching a movie of the process run forwards and then backwards, we would not be able to distinguish the two.   The process is stationary if for every $\tau \in \R$ the process $\{X_t\}$ and the process $\{X_{t+\tau}\}$ have the same finite dimensional distributions.  A time- reversible process is stationary but the converse is not true in general. Importantly to our purposes, a stationary diffusion process with invariant measure $\mu$ is reversible if and only if its generator is self-adjoint in $\ltm$ (see \cite[Theorem 4.5]{Pavliotis}) i.e. iff
	\be\label{detbalance}
	\int_{\R^d} f(x) (\cL g)(x) d\mu(x) = \int_{\R^d} g(x) (\cL f)(x) d\mu(x), \quad \mbox{for every } f,g, \in \cD_2(\cL). 
	\ee
	When the above holds we also say that $\mu$ satisfies {\em detailed balance} with respect to $\cL$. If $\cL$ generates a strongly continuous contraction semigroup, by the Lumer-Phillips Theorem \cite[Section 1.4]{Pazy} it is negative as well hence, 
	with a calculation completely analogous to the one leading to \eqref{dissipationofnorm}, one can see that  reversible dynamics are dissipative. However in the acronym GENERIC -- General Equations for Non-Equilibirum Reversible Irreversible coupling -- the term ``reversible"  refers to a different property, enjoyed in particular by Hamiltonian dynamics, which are certainly non-reversible according to the definition we have given above. Indeed, Hamiltonian dynamics are associated with the antisymmetric part of the evolution and hence they are conservative. A simple example can be given by considering planar  rotations (which are Hamiltonian dynamics), for which one  can  distinguish whether the process is running forwards or backward in time by looking at the verse of rotation.  In the acronym GENERIC the term  ``reversible'' refers to flows  $\phi_t$ on $\R^{2d}$, $\phi_t:\R^{2d} \rightarrow \R^{2d}$,  which enjoy the following property:
	$$
	(\phi_t)^{-1}(x,v)=\phi_{-t}(x,v) =(F \circ \phi_t \circ F)(x,v), \quad \mbox{where } F(x,v) = (x, -v)\,,
	$$
	while the word ``irreversible" alludes to the {\em macroscopic} irreversibility of the dynamics, i.e. to dissipativity.  We will avoid using the word ``reversible" at all, but when we do we will mean it in the sense of microscopic time-reversibility.

	
	\section{Relating linear hypocoercivity and Wasserstein pre-GENERIC}\label{sec:3}
	\subsection{From linear hypocoercivity to Wasserstein pre-GENERIC}\label{linhypoc to quadpre-Gen}
	Suppose the m-KE is in linear hypocoercive form \eqref{m-KEactual}, with both $B$ and $B'$ obeying chain and product rule.   In this section we  show that, if this is the case, then  the Fokker-Planck equation for $\rho_t=h_t \mu$ is in Wasserstein pre-GENERIC form \eqref{quadratic pre-generic} with entropy functional $\cS(\rho)$ given by the {\em relative entropy} $\cS_{\mu}(\rho)$ of $\rho$ with respect to $\mu$, i.e.
	\be\label{hypto GEN definitions}
	\cS(\rho) = \cS_{\mu}(\rho):=\int_{\R^d} \log \left( \frac{\rho(x)}{\mu (x)}\right) \rho(x) dx,
\end{equation}
and 
\be\label{hypto GEN definitions2}
M_{\rho} (\xi) = 2 A'(\rho A \xi), \quad W(\rho)= -B \rho\, , 
\ee
where the operators $A$ and $B$ in the above are as prescribed by the hypocoercive structure \eqref{m-KEactual}. 
Indeed, if the m-KE is in the form \eqref{m-KEactual} then 
\be\label{put1}
\partial_t\rho_t=\mu \partial_t h_t=-\mu (B+A^*A)h_t=-\mu B(\rho_t/\mu)-\mu A^*A(\rho_t/\mu).
\ee
For the first term in \eqref{put1}, using the fact that the kernel of $\cL$ is made of constants  and the first equality in \eqref{xx},   we get 
$$
-\mu B(\rho_t/\mu)= -\mu Bh_t = B'(h_t \mu) = -B\rho_t \, ,
$$
hence
\be\label{brho}
\mu B(\rho_t /\mu ) = B \rho_t \,.  
\ee
For the second term in \eqref{put1}, we can see that for every $f \in \dense$ one has
\begin{align}
	\label{eq: integral 1}
	\int_{\R^d} \mu(x) [A^*A(\rho/\mu)](x) f(x) dx &=\int_{\R^d}  [A(\rho/\mu)](x) (A f)(x) \mu(x) dx \nonumber \\
	&=\int_{\R^d} [A'(A(\rho/\mu)\mu)](x) f(x) dx ;
\end{align}
using the fact that $A$  satisfies the chain rule, we can further express the inner integral in \eqref{eq: integral 1} as follows
$$
A(\rho/\mu)=(\rho/\mu) A\left(\log(\rho/\mu)\right).
$$
Substituting the above back into \eqref{eq: integral 1} we then get
$$
\int_{\R^d} \mu(x) [A^*A(\rho/\mu)](x) f(x) dx=\int_{\R^d} A'\left[ \rho A\left(\log(\rho/\mu)\right)\right](x) f(x) dx \,, \quad \mbox{for every } f \in \dense. 
$$
Now notice that if $\cS$ is as in \eqref{hypto GEN definitions} then $\vd \cS (\rho) = \log\left({\rho}/{\mu} \right) +1$. Again by using the fact that $\mathrm{ker} \cL = \{constants\}$ and \eqref{kerl-kerB int KerA}, we have $A \mathbf{1}=0$, hence  $A(\vd \cS (\rho))= A\left(\log\left({\rho}/{\mu} \right) \right)$, from which we deduce
\be\label{put2}
\mu A^*A(\rho/\mu)=A'(\rho A\log(\rho/\mu))=
A'\left[\rho A \left(\vd\cS(\rho)\right)\right] = M_{\rho}\left(\frac{1}{2}\vd\cS(\rho)\right)\,.
\ee
Putting together \eqref{put1}, \eqref{brho} and \eqref{put2}, we have therefore obtained 
\begin{align*}
	\partial_t\rho_t =\cL'\rho_t
	& = -B\rho_t - A'(\rho A(\vd \cS(\rho)))\\
	&= W\rho_t-M_{\rho_t}\left(\frac{1}{2}\vd\cS(\rho_t)\right)\,.
\end{align*}

To show that the above is a pre-GENERIC system we still need to  verify that the degeneracy condition \eqref{degconpre-GENERIC} holds; to this end,  using  the chain rule for $B$, $B{\bf 1}=0$ and $B^*=-B=B'$ (which follows from Lemma \ref{lemma:B*=B'=-B}), we have 
\begin{align}\label{kerBorthogonality}
	\int_{\R^d} (B\rho)(x) [\vd\cS(\rho)](x) dx& = \int_{\R^d} (B \rho)(x) [\log(\rho(x)/\mu(x)) +1] dx \nonumber \\
	& = - \int_{\R^d} \mu(x) [B(\rho /\mu)](x) dx
	- \int_{\R^d} \rho(x) B1 \, dx \nonumber\\
	& = -\int_{\R^d} (B^* \mu)(x) \rho(x)/\mu(x) dx = \int_{\R^d} (B \mu)(x) \rho(x)/\mu(x) dx = 0 \,. 
\end{align} 
It remains to check that $M$ is symmetric and positive semidefinite. In fact,
$$
( M\prho\xi, \xi)=2\left( A'(\rho A \xi),  \xi \right)=2\int_{\R^d} \rho(x) [(A\xi)^2(x)] dx\geq 0,  \quad \mbox{for every } \xi \in \dense, 
$$
as $\rho>0$ (standing assumption \ref{SA3}).
Similarly, for every  $\xi_1, \xi_2 \in \dense$, we have 
$$
(M\prho\xi_1, \xi_2) =2\int_{\R^d} [A'(\rho A \xi_1)](x) \xi_2(x) dx=2\int \rho(x) (A\xi_1)(x) (A\xi_2)(x) dx= ( M\prho\xi_2, \xi_1 )\,.  
$$
\subsection{From Wasserstein pre-GENERIC to linear Hypocoercivity}\label{sec:from pre-Gen to hyp}
Recalling that  $\mu$ is the unique invariant measure of the underlying process $X_t$,  we now assume that the FP equation  is in pre-GENERIC form \eqref{degconpre-GENERIC}-\eqref{quadratic pre-generic}, with $\cS(\rho)=\cS_{\mu}(\rho)$ as in \eqref{hypto GEN definitions} and $W$ is  an operator on $\dense$ such that both $W$ and $W'$ are defined on $\dense$ and obey the chain and product rule; because $W'$ obeys the product rule,   the kernel of $W'$ contains constants, i.e. $W'{\bf{1}}=0$.
Assuming that that FP equation is of the form \eqref{quadratic pre-generic} means that the FP operator $\cL'$ is precisely given by
$$
\cL' \rho=  W \rho -M_{\rho}\left(\frac{1}{2}\vd \cS(\rho)\right)\,.
$$
With this premise, we will show that the dual of $\cL'$, i.e. the Kolmogorov  operator $\cL$,\footnote{We are implicitly assuming that $(\cL')'=\cL$, at least on $\dense$,  similarly for $B$ and $W$.}
can be written in linear hypocoercive form $\cL= B-A^*A$. In particular we will show that $W'$ is antisymmetric in $\ltm$ and it is related to the operator $B$ appearing in the hypocoercive form  by $W'=B$ while $M_{\rho}$ is related to $A$ through the relations \eqref{squareroot}-\eqref{relMA} below.    

\smallskip
We start by showing that $W'$ is antisymmetric. To this end, from the degeneracy condition \eqref{degconpre-GENERIC} 
we have 
\begin{align}\label{ortogdegen2}
	0 = (W(\rho), \vd\cS(\rho)) = \int_{\R^d} \!\!\rho(x) \, [W'(\log(\rho/\mu))](x) dx+ \int_{\R^d} \!\!\rho(x)\,  (W'{\bf 1}) dx = \int_{\R^d} \frac{\rho(x)}{\mu(x)} (W \mu)(x) dx;    
\end{align}
since the degeneracy condition \eqref{degconpre-GENERIC} holds for every $\rho\in \dense$, this implies 
$W \mu =0$.  
Using Lemma \ref{lemma:B*=B'=-B} we conclude that  $B=W'$ is antisymmetric in $\ltm$. Let us notice that by setting $B=W'$ we also have $B {\bf{1}}=0$. 

Regarding the gradient-flow part, because  $M_{\rho}$ is symmetric and positive definite for every $\rho > 0$ fixed,  one can find a square root of $M_{\rho}$, i.e. an operator $M_{\rho}^{1/2}$ such that $(M_{\rho}^{1/2})'M_{\rho}^{1/2} = M_{\rho}$; guided by the calculation in the previous subsection, we assume that one such square root\footnote{As is well known, the square root operator is not unique.} can be written in the form
\be\label{squareroot}
M_{\rho}^{1/2}(\xi):= \sqrt{2\rho} A (\xi), 
\ee
for some  operator $A$ which does not depend on $\rho$ (let us point out that any It\^o diffusion belongs to this category, see Section \ref{Diffusion processes}, and that expressions analogous to the above also appear in \cite{ZimmerCelia2014}). 
With this choice, we have 
$
(M_{\rho}^{1/2})'(\xi) = \sqrt{2}A'(\sqrt{\rho} \xi)$, so that 
\be\label{relMA}
M_{\rho}(\xi) = (M_{\rho}^{1/2})'(M_{\rho}^{1/2}(\xi)) = 2A'(\rho A(\xi)). 
\ee
If $A$  obeys the chain rule then,  with calculations completely analogous to those in the previous subsection, one finds that if $\rho_t$ solves \eqref{quadratic pre-generic} and all of the above assumptions are satisfied, then the m-KE for $h_t=\rho_t/\mu$ can be written in the form \eqref{m-KEactual}.

To conclude, recall that the fact that the kernel of $\cL$ is only made of constants is a consequence of ergodicity, see Section \ref{subsec:setup}. However, in this setup this can anyway be recovered from the structure of the equations. Indeed,  because $W\mu=0$ and $\mu$ is an invariant measure, i.e. a stationary solution of \eqref{quadratic pre-generic}, it follows that $M_{\mu}(\vd\cS(\mu))=0$; if this is the case then, since $\vd\cS(\mu)=1$,  for every $g \in \dense$ one has
$$
0= (M_{\mu}(\vd\cS(\mu)), g) = (M_{\mu}^{1/2}(\vd\cS(\mu)), M_{\mu}^{1/2}(g)) = \int (A{\bf 1}) (Ag) \mu =  \int  (A^*A {\bf{1}}) \, g \mu \, ,
$$
which implies ${\bf{1}} \in \ker(A^*A)$. Because $\ker (A^*A)=\ker A$, one has $A{\bf{1}}=0$. This, together with the fact that $B {\bf{1}} = W' {\bf{1}} =0$ implies $\cL {\bf{1}} = \mko {\bf{1}} =0$. 

\begin{Note}\label{note:linghyptopregen}
	A couple of observations on Section \ref{linhypoc to quadpre-Gen} and Section \ref{sec:from pre-Gen to hyp}. 
	\begin{itemize}	
		\item The calculation in \eqref{kerBorthogonality}  shows that, in the framework of  Section \ref{linhypoc to quadpre-Gen}, i.e. if the generator $\cL$ is in linear hypocoercive form,  the degeneracy condition \eqref{degconpre-GENERIC} is ultimately a consequence of the antisymmetry in $\ltm$ of $B$ and of $B \textbf{1}=0$, which implies $B \mu=0$ (by Lemma \ref{lemma:B*=B'=-B}). Viceversa, in Section \ref{sec:from pre-Gen to hyp} the degeneracy condition implies $W \mu =0$ (hence $B \mu =0$ as well), but it does not seem to imply $W'1=B1=0$ directly.   
		\item On a practical level one can see the content of these subsections as giving explicit formulas to go  from linear hypocoercivity to pre-GENERIC formulation and viceversa. The potential advantage is that one could obtain  some results almost for free; for example, suppose a particular equation is known to have hypocoercive form. Using Section \ref{linhypoc to quadpre-Gen} one can then write it in pre-GENERIC form and hence obtain relative entropy dissipation.
		An indeed a {\bf Corollary} of the content of  Section \ref{linhypoc to quadpre-Gen} is the following: if the generator $\cL$ is in linear hypocoercive form, then the relative entropy is dissipated along the flow of the associated FP equation. 
		\item When $B=0$, a straightforward consequence of Section \ref{linhypoc to quadpre-Gen} is the following: if the generator of the process is self-adjoint in $\ltm$, which is equivalent to the underlying process being reversible, the the Fokker Planck operator is in gradient flow form. In this sense the content of  Section \ref{linhypoc to quadpre-Gen} can be seen as providing a simple proof of the link between reversibility and gradient flow structure. 
		\end{itemize}
	\end{Note}

\section{Relation between pre-GENERIC and symmetric-antisymmetric decomposition}\label{section:decomposition}
In Section \ref{sec:3} we worked with linear Markov generators in hypocoercive form \eqref{m-KElinearhypocform}. 
In this section we slightly relax the class of Markov generators we consider; we work again under the standing assumptions of Hypothesis \ref{SA} but this time we simply split the generator $\cL$ as follows
\begin{equation}\label{eq: decomposition L}
	\mathcal{L}=\mathcal{L}_a+\mathcal{L}_s,    
\end{equation}
where  $\mathcal{L}_s$ and $\mathcal{L}_a$ are, respectively, the symmetric and anti-symmetric parts in  $\ltm$ of $\mathcal{L}$. We will assume that $\cL_s$ is linear, but no such assumption will be needed for $\cL_a$. Moreover, we don't explicitly assume that $\cL_s$ should be of the form $-A^*A$, as in Section \ref{sec:3}. {\footnote{ We note that for less than regularity issues related to the domain of the adjoint operator and for the potential practical difficulty of finding $A$,  any non-negative symmetric operator can be written in the form $-A^*A$. But this might not be easy in practice and indeed this is inconvenient for the PDMPs we study in Section \ref{sec:PDMPs}.}} With the splitting \eqref{eq: decomposition L} in mind,  for every $\rho \in \dense$ we consider the following Hamiltonian functional  $\mathscr{H}(\rho; \cdot):\cD(\mathscr{H}_{\rho})\rightarrow \R$,    
\begin{align}
	\label{eq: decompose H}
	\mathscr{H}(\rho;\xi)&:=\int_{\R^d} e^{-\xi(x)} \left(\mathcal{L} e^{\xi}\right)\!(x)\rho(x) dx\\
	&=\int_{\R^d} e^{-\xi(x)}  \left[(\mathcal{L}_a+\mathcal{L}_s) e^{\xi}\right]\!(x)\rho(x) dx
	=:\mathscr{H}_a(\rho;\xi)+\mathscr{H}_s(\rho;\xi),
	\nonumber 
\end{align}
where, for every $\rho \in \dense$, $\mathscr{H}_s(\rho; \cdot):\cD(\mathscr{H}_{s,\rho})\rightarrow \R$, $\mathscr{H}_a(\rho; \cdot):\cD(\mathscr{H}_{a,\rho})\rightarrow \R$, 
\be\label{HaHs}
\mathscr{H}_a(\rho;\xi):=\int_{\R^d} e^{-\xi} \left(\mathcal{L}_a e^{\xi}\right)d\rho,\quad \mathscr{H}_s(\rho;\xi):=\int_{\R^d} e^{-\xi} \left(\mathcal{L}_s e^{\xi}\right)d\rho \, ,
\ee
and $\cD(\mathscr{H}_{\rho})$ is the set of functions $\xi$ such the integrand in \eqref{eq: decompose H} makes sense  and  the integral in \eqref{eq: decompose H} is finite (similarly for $\cD(\mathscr{H}_{s,\rho}), \cD(\mathscr{H}_{a,\rho})$). We will also make the following technical assumption. 
\begin{assumption}\label{ass:domH}
	For every $\rho \in \dense$, $\cD(\mathscr{H}_{s,\rho})$ is such that for any  $\xi, \bar\xi, \tilde \xi \in \cD(\mathscr{H}_{s,\rho})$, $\xi+\gamma \bar \xi \in \cD(\mathscr{H}_{s,\rho})$ for any $\gamma \in (0,1)$ and  the  integral 
	$$
	\int_{\R^d} e^{-\xi(x)}  \alpha\xi^p(x) \left[\cL_s (\beta\bar\xi^p e^{\tilde \xi})\right](x) \rho(x) \, dx
	$$
	is finite,  for any  $p \in \{1,2\}, \alpha, \beta \in \{0,1\}$.  
\end{assumption}
The above assumption is easily satisfied in practice, as in examples $\cD(\mathscr{H}_{s,\rho})$  will typically contain functions that grow polynomially while $\rho$ decays exponentially fast at infinity. Finally, before stating our main theorem, we recall (see \eqref{equivdefpre-generic}) that the  Fokker-Planck equation $\dot{\rho_t}=\mathcal{L}'(\rho_t)$  is in pre-GENERIC form  with respect to $(\cS,\psi, \psi^{\star},W)$ if and only if 
\begin{equation}
	\label{eq: def of preGENERIC}
	\psi(\rho_t;\mathcal{L}'(\rho_t)-W(\rho_t))+\psi^{\star}\left(\rho_t;-\frac{1}{2}\vd \cS(\rho_t)\right)+\frac{1}{2}( \mathcal{L}'(\rho_t), \vd \cS(\rho_t)) =0 
\end{equation}
and the orthogonality condition \eqref{degconpre-GENERIC} holds. 

{With this in mind, the purpose of Theorem \ref{theorem:decomposition} below is to show that if the symmetric-antisymmetric decomposition  \eqref{eq: decomposition L} holds for the generator then the Fokker-Planck equation can be written in pre-GENERIC form, and vice versa. The theorem also provides a constructive way to pass from one formulation to the other. }

\begin{theorem}\label{theorem:decomposition} With the notation introduced so far, the following holds: 

{ {\bf $\bullet$ From symmetric-antisymmetric decomposition of the generator to pre-GENERIC structure of the Fokker-Planck equation}.}
	Let  $\cL$  be the generator of a time-homogeneous Markov process  with  unique invariant measure $\mu$ and  consider a symmetric-antisymmetric splitting of $\cL$ as in \eqref{eq: decomposition L}
	where we  assume that $\cL_s$ is a linear operator.    Given $\cL$, we consider $\mathscr{H}_s$ as in \eqref{HaHs} and assume that $\mathscr{H}_s (\rho; \cdot)$ is {convex} and  once Frechet differentiable in the second argument, for every $\rho$.  
	For every $\rho \in \dense$ we then let $\Psi^{\star}(\rho; \cdot):\cD(\mathscr H_{s,\rho})\rightarrow \R$ be the functional   
	\begin{equation}
		\label{eq: Psi from Ls}
		\Psi^{\star}(\rho;\xi)=\mathscr{H}_s\left(\rho; \frac{1}{2}\vd \cS_{\mu}(\rho)+\xi\right)-\mathscr{H}_s\left(\rho; \frac{1}{2}\vd \cS_{\mu}(\rho)\right)\, ,
	\end{equation}
	where $\cS_{\mu}$ is the relative entropy between $\rho$ and $\mu$. 
	Suppose  $\cL_s {\bf 1}=0$ and let  Assumption \ref{ass:domH}  and   the following orthogonality condition hold, 
	\be\label{ortogproof}( \mathcal{L}_a'(\rho), \vd \cS_{\mu}(\rho))=0 \, \quad \mbox{for every } \rho \in \dense \,. 
	\ee 
	Then the Fokker-Planck equation $\dot{\rho}=\mathcal{L}'(\rho)$ is a $(\cS_{\mu},\Psi, \Psi^{\star},\mathcal{L}_a')$ pre-GENERIC flow.
	That is, the operator $\cL_s$ is associated with the gradient flow part of the dynamics, i.e.
	\be\label{lsprimoproof}
	\cL_s'\rho = \vd_{\xi} \Psi^{\star}\left(\rho; -\frac{1}{2} \vd \cS_{\mu}(\rho)\right)
	\ee
	and  the flow generated by $\cL'$ can be written as 
	\be\label{splitting1}
	\pa_t \rho_t = W (\rho_t)+\vd_{\xi} \Psi^{\star}\left(\rho_t; -\frac{1}{2} \vd \cS_{\mu}(\rho_t)\right), \quad \mbox{with } W\rho = \cL_a'\rho \,.
	\ee
	
{ {\bf $\bullet$ From pre-GENERIC form of the Fokker-Planck equation to symmetric-antisymmetric splitting of the generator}.} Vice versa, suppose that there exist $\psi, \psi^{\star},W$ and $\cS \in C^1$ such that the Fokker-Planck equation $\dot{\rho}_t=\mathcal{L}'(\rho_t)$ is a $(\cS,\psi, \psi^{\star},W)$ pre-GENERIC system;  that is, suppose 
	\be\label{splitting2}
	\pa_t \rho_t =\cL'\rho_t =  W (\rho_t)+\vd_{\xi} \psi^{\star}\left(\rho_t; -\frac{1}{2} \vd \cS(\rho_t)\right)\,,
	\ee
	with $W$ and $\cS$ such that $(W(\rho), \vd \cS(\rho))=0$ for every $\rho \in \dense$.  
	Let 
	$$
	\tilde{\mathscr{H}}(\rho;\xi):= \psi^{\star}\left(\rho; \xi- \frac12 \vd\cS(\rho)\right) - \psi^{\star}\left(\rho; -\frac12 \vd\cS(\rho)\right)
	$$
	and assume that $\tilde{\mathscr{H}}$ is convex and $\tilde{\mathscr{H}}(\rho;0)=0$ for every $\rho \in \dense$. Assume moreover that $\tilde{\mathscr{L}}$, the Legendre transform of $\tilde{\mathscr{H}}$,  is twice differentiable at the point $(\mu;0)$. 
	Then the operator $\mathcal{L}-W'$ is  symmetric with respect to $L^2_\mu$ and
	$$
	\mathcal{L}=(\mathcal{L}-W')+W'
	$$
	constitutes a symmetric-antisymmetric decomposition of $\mathcal{L}$. 
\end{theorem}
\begin{proof}
	{We first show that $\Psi^{\star}$ defined in \eqref{eq: Psi from Ls} is a dissipation potential, i.e. that it is non-negative, symmetric, convex in the second argument and attains its minimum at $0$. In fact, the convexity of $\Psi^{\star}$ follows immediately from the convexity of $\mathscr{H}_s$. We will show in   Note \ref{sec: alternative computation for H} (take $\trans$ to be the identity there) that the symmetry in $\ltm$ of  $\mathcal{L}_s$ implies  
		\begin{equation}
			\label{eq: symmetry of Hs}
			\mathscr{H}_s(\rho;\frac{1}{2}\vd\cS_\mu(\rho)+\xi)=\mathscr{H}_s(\rho;\frac{1}{2}\vd\cS_\mu(\rho)-\xi),    
		\end{equation}
		which in turn  implies the symmetry of $\Psi^{\star}$. From the definition  of $\Psi^{\star}$, equation \eqref{eq: Psi from Ls}, we have  
		$$
		\Psi^{\star}(\rho,0)=0\quad\text{and}\quad\vd_\xi\Psi^{\star}(\rho;\xi)=(\vd_\xi \mathscr{H}_s)\left(\rho;\frac{1}{2}\vd \cS_{\mu}(\rho)+\xi\right),
		$$
		thus, 
		$$
		\vd_\xi\Psi^{\star}(\rho;0)=(\vd_\xi \mathscr{H}_s)\left(\rho;\frac{1}{2}\vd \cS_{\mu}(\rho)\right)\quad\text{and}\quad	\vd_\xi\Psi^{\star}\left(\rho;-\frac{1}{2}\vd \cS_{\mu}(\rho)\right)=(\vd_\xi \mathscr{H}_s)(\rho,0).
		$$
		To clarify notation, the RHS of the above formulas is the Frechet derivative of $\mathscr{H}_s$ with respect to its second argument, calculated at $\frac{1}{2}\vd \cS_{\mu}(\rho)$ and $0$ respectively. 
		By taking the derivative with respect to $\xi$ on both sides of \eqref{eq: symmetry of Hs} and then letting $\xi=0$, we get
		$$
		\vd_\xi\Psi^{\star}(\rho;0)=(\vd_\xi \mathscr{H}_s)\left(\rho;\frac{1}{2}\vd \cS_{\mu}(\rho)\right)=0.
		$$
		Using the convexity and differentiability of $\Psi^{\star}$, we have
		$$
		\Psi^{\star}(\rho;\xi)\geq \Psi^{\star}(\rho;0)+(\vd_\xi\Psi^{\star}(\rho;0),\xi)=0, \,\, \mbox{ for every } \rho, \xi \,.
		$$
		We have thus proved that $\Psi^{\star}$ is a dissipation potential. Next we will establish \eqref{splitting1}.}
	For every $\bar \xi \in \cD(\mathscr{H}_{s,\rho})$, we have 
	\be\label{limit}
	( \vd_{\xi} \mathscr{H}_s(\rho;\xi),\bar{\xi})=\lim\limits_{\varepsilon\rightarrow 0}\frac{\mathscr{H}_s(\rho; \xi+\varepsilon \bar{\xi})-\mathscr{H}_s(\rho;\xi)}{\varepsilon}=\int_{\R^d}\Big[-\bar{\xi}e^{-\xi} \mathcal{L}_s e^\xi+ e^{-\xi}\mathcal{L}_s(\bar{\xi}e^{\xi})\Big]\,d\rho.
	\ee 
	We postpone showing how the above expression is obtained and proceed with the main argument. From the above, 
	$$
	( \vd_{\xi} \mathscr{H}_s(\rho;0),\bar{\xi})=\int_{\R^d}\Big[-\bar{\xi} (\mathcal{L}_s \mathbf{1})+ (\mathcal{L}_s\bar{\xi})\Big]\,d\rho=\int_{\R^d} (\mathcal{L}_s\bar{\xi})\,d\rho=(\mathcal{L}_s'(\rho),\bar{\xi}),
	$$
	where we have used $\mathcal{L}_s{\mathbf{1}}=0$. From this, we deduce 
	$$
	\mathcal{L}_s'(\rho)=\vd_\xi \mathscr{H}_s(\rho;0)=
	\vd_\xi\Psi^*\left(\rho;-\frac{1}{2}\vd\cS_{\mu}(\rho)\right)\, , 
	$$
	which is precisely \eqref{lsprimoproof}. That is,  $\cL_s'$ is in gradient flow form; using \eqref{gradflowequivdefinition} and writing $\mathcal{L}_s'(\rho)=\mathcal{L}'(\rho)-\mathcal{L}_a'(\rho)$, we then have
	\begin{align*}
		\Psi(\rho; \mathcal{L}'(\rho)-\mathcal{L}_a'(\rho))+\Psi^{\star}(\rho;-\frac{1}{2}\vd \cS(\rho))=\left(\mathcal{L}'(\rho)-\mathcal{L}_a'(\rho), -\frac{1}{2}\vd  \cS(\rho)\right) \stackrel{\eqref{ortogproof}}{=}\left( \mathcal{L}'(\rho), -\frac{1}{2}\vd \cS(\rho)\right), 
	\end{align*}
	which,  is equivalent to \eqref{splitting1} (see \eqref{equivdefpre-generic}). 
	We are now only left with showing the limit \eqref{limit}. To this end, for some $\theta=\theta(x) \in (0,1)$ we have:
	\begin{align*}
		\mathscr{H}_s(\rho; \xi+\varepsilon \bar{\xi})&=\int_{\R^d} e^{-(\xi+\varepsilon \bar{\xi})}\big[\mathcal{L}_s (e^{\xi+\varepsilon \bar{\xi}})\big]\,d\rho
		\\&=\int_{\R^d} \Big(e^{-\xi}-\varepsilon \bar{\xi}e^{-\xi}+\frac{1}{2} \varepsilon^2 
		(\bar \xi(x))^2 e^{-\xi(x)- \theta \varepsilon \bar\xi(x)}\Big)\big[\mathcal{L}_s \big(e^{\xi}+\varepsilon \bar{\xi} e^{\xi}+\frac{1}{2} \varepsilon^2 
		(\bar \xi(x))^2 e^{\xi(x)+ \theta \varepsilon \bar\xi(x)}\big)\big]\,d\rho
		\\&=\int_{\R^d}\Big(e^{-\xi}\mathcal{L}_s e^\xi+\varepsilon \big[-\bar{\xi}e^{-\xi}\mathcal{L}_s(e^\xi)+e^{-\xi}\mathcal{L}_s(\bar{\xi}e^\xi)\big]\Big)\,d\rho + \varepsilon^2 I
		\\& = \mathscr{H}_s(\rho,\xi) + 
		\varepsilon \int_{\R^d} \big[-\bar{\xi}e^{-\xi}\mathcal{L}_s(e^\xi)+e^{-\xi}\mathcal{L}_s(\bar{\xi}e^\xi)\big]\,d\rho + \varepsilon^2 I\, ,		
	\end{align*}
	where the term $I$ gathers all the terms which are multiplied by a power of $\varepsilon$ bigger or equal than two. In view of Assumption \ref{ass:domH} all the integrals appearing in the above are finite, hence we can take the limit $\varepsilon \rightarrow 0$ and obtain the result.

	Now suppose that \eqref{splitting2} holds true or, equivalently, 
	\begin{equation}
		\label{eq: HT from GENERIC}
		\psi(\rho_t;\mathcal{L}'(\rho_t)-W(\rho_t))+\psi^{\star}(\rho_t; -\frac{1}{2}\vd \cS(\rho_t))+\frac{1}{2}( \mathcal{L}'(\rho_t), \vd\cS(\rho_t))=0,  \quad  ( W(\rho),
		\vd \cS(\rho)) =0 \,\,\,\forall \rho \in \dense \,.
	\end{equation}
	By subtracting $\frac{1}{2}\left( W(\rho_t), \vd\cS(\rho_t)\right)=0$ from the LHS of  \eqref{eq: HT from GENERIC}, we get
	\begin{equation}\label{subtrW}
		\psi(\rho;\mathcal{L}'(\rho)-W(\rho))+\psi^{\star}(\rho;-\frac{1}{2}\vd \cS(\rho))+( \mathcal{L}'(\rho)-W(\rho), \frac{1}{2} \vd\cS(\rho))=0.
	\end{equation}
	{The above implies that the operator $\cL'-W$ is a gradient flow. From
		\cite[Theorem 3.4]{kpelmath} the Hamiltonian $\tilde{\mathscr{H}}$ is reversible with respect to $\cS$, i.e. $\tilde{\mathscr{H}}(\rho; \xi)= \tilde{\mathscr{H}}(\rho;\vd\cS(\rho)-\xi)$ for every $\xi$.  This is equivalent to $\tilde{\mathscr{L}}$, the Legendre transform of $\tilde{\mathscr{H}}$,  satisfying equality \eqref{eq: relationLdetbalance}  (that is, the relation \cite[(2.6)]{MielkePeletierRenger2014}).  This implies, by \cite[Theorem 3.3]{MielkePeletierRenger2014}, to the fact that the operator $(\cL'-W)'$ satisfies detailed balance with respect to $\mu$}. This concludes the proof.
\end{proof}
\begin{Note}
\label{Note on decomposition}
	Some comments on Theorem \ref{theorem:decomposition}. 
	\begin{itemize}
		\item The first part of the theorem implies that if $\cL_s$ is symmetric in $\ltm$ then $\cL'_s$ is in gradient flow form. We prove this fact without making use of LDPs. It must be  emphasized that we find {\em one possible} gradient flow structure for $\cL'_s$, the one that corresponds to using the relative entropy as entropy functional. In other words the first part of the theorem says that we can always look at $\cL'_s$ as a gradient flow for the relative entropy. It is this ``a-priori choice" of entropy functional that allows one to bypass the use of LDPs. Avoiding the use of LDPs makes proofs simpler, and this is particularly true in the linear case of Section \ref{sec:3}, where only straightforward arguments are used. However it also conceals important physical considerations (summarised in the introduction) which shed a more profound light on the microscopic origin of the gradient flow structure.    
		\item Condition \eqref{ortogproof} can be dropped by using the observation in Note \ref{note:variousondissip}, point ii).
		\item In the proof of the second part of the theorem we used \cite[Theorem 3.3]{MielkePeletierRenger2014}, which assumes the validity of a LDP. However the proof of the specific implication of that theorem that we use  here does not require assuming the validity of a LDP. 
		\item Suppose that $\mathscr{H}_s(\rho;\cdot)$ is differentiable. Then it is convex in the second argument, for every $\rho$ fixed, if and only if
			$$
			\mathscr{H}_s(\rho;\xi+\bar{\xi})\geq \mathscr{H}_s(\rho;\xi)+( \vd_{\xi} \mathscr{H}_s(\rho;\xi),\bar{\xi})\, ,
			$$
			for all functions $\xi,\bar{\xi}$ in the appropriate domains.  
			Let $\Gamma_s$ be the carr\'{e} du champ operator associated to $\mathcal{L}_s$, namely
			$$
			\Gamma_s(\xi,\bar{\xi}):=\frac{1}{2}\left(\mathcal{L}_s(\xi \bar{\xi})-\xi\mathcal{L}_s\bar{\xi}-\bar{\xi}\mathcal{L}_s\xi\right).
			$$
			Then
			\begin{align*}
				\mathscr{H}_s(\rho;\xi+\bar{\xi})&=\int_{\R^d}\left(e^{-\xi}\mathcal{L}_s e^{\xi}+e^{-\bar{\xi}}\mathcal{L}_s e^{\bar{\xi}}+2e^{-(\xi+\bar{\xi})}\Gamma_s(e^{\xi},e^{\bar{\xi}})\right)\,d\rho(x),
				\\ \mathscr{H}_s(\rho;\xi)+( \vd_{\xi} \mathscr{H}_s(\rho;\xi),\bar{\xi})&=\int_{\R^d} e^{-\xi}\mathcal{L}_s e^{\xi}\,d\rho(x)+\int \left(\mathcal{L}_s\bar{\xi}+2e^{-\xi}\Gamma_s(e^{\xi},\bar{\xi})\right)\,d\rho(x).
			\end{align*}
			Thus $\mathscr{H}_s$ is convex if and only if
			\begin{equation}
				\label{eq: convexity condition}
				\int_{\R^d} e^{-\bar{\xi}}\left(\mathcal{L}_s e^{\bar{\xi}}+ 2e^{-\xi}\Gamma_s(e^{\xi},e^{\bar{\xi}})\right)\,d\rho(x)\geq \int_{\R^d}\left(\mathcal{L}_s \bar{\xi}+2e^{-\xi}\Gamma_s(e^{\xi},\bar{\xi})\right)\,d\rho(x)~\text{for all}~ \xi,\bar{\xi}.
			\end{equation}
			For diffusion processes, $\mathcal{L}_s$ and $\Gamma_s$ satisfy the chain rule 
			$$
			\mathcal{L}_s\phi(\xi)=\partial_\xi\phi(\xi)\mathcal{L}_s \xi+ \partial^{2}_\xi\phi(\xi)\Gamma_s(\xi,\xi)\quad\text{and}\quad \Gamma_s(\phi(\xi),\bar{\xi})=\partial_\xi\phi(\xi)\Gamma_s(\xi,\bar{\xi}),
			$$
			thus
			\begin{equation}
			    \label{eq: diffusion-carreduchamp}
					\mathcal{L}_s e^{\bar{\xi}}=e^{\bar{\xi}}(\mathcal{L}_s \bar{\xi}+\Gamma_s(\bar{\xi},\bar{\xi})),\quad e^{-(\xi+\bar{\xi})}\Gamma_s(e^{\xi},e^{\bar{\xi}})=e^{-\xi}\Gamma_s(e^{\xi},\bar{\xi})=\Gamma_s(\xi,\bar{\xi}).
				\end{equation}
			In this case the convexity condition \eqref{eq: convexity condition} reduces to 
			$$
			\int_{\R^d} \Gamma_s(\bar{\xi},\bar{\xi})\, d\rho(x)\geq 0,
			$$
			which is always true as $\Gamma_s(\bar{\xi},\bar{\xi}) \geq 0$ (see for instance \cite[Section 1.4.2]{bakry2013analysis}).
		\end{itemize}
\end{Note}


\section{Large Deviation Principles (LDPs) and  generalised-reversibility}\label{sec:4}
In this section we  consider a particular class of non-reversible processes, the class of so-called  {\em generalised-reversible} processes, and we characterize the form of their large deviation rate functional. To this end, in Subsection \ref{subsec:4.1}  and Subsection \ref{subsec:4.2}, respectively,  we give a summary background on generalised reversibility and  LDPs, respectively. In Subsection \ref{subsec:4.3} we state and prove the main result of this section,  Proposition \ref{prop:irrevlagrangian}. 
\subsection{Generalised Reversibility}\label{subsec:4.1}
We start by briefly recalling the definition and main facts  about generalised reversibility, more details can be found in \cite[Section 2.2.1.2]{Lelievre2010}, but this definition goes back to at least  the work of Yaglom \cite{Yaglom} (and for this reason generalised reversible processes are also called Yaglom-reversible at times).
Let $\mu$ be a probability measure on $\R^d$ and $\trans$ be an involutive transformation (i.e. $\trans = \trans^{-1}$) on $\R^d$ leaving $\mu$ invariant, that is,  
\be\label{invarS}
\mu(\trans (dx)) = \mu(\trans^{-1} (dx)) = \mu (dx) \,.
\ee

An $\R^d$-valued time-homogeneous  Markov process $(X(t))_{t\geq 0}$ is said to be {\em generalised reversible with respect to $\mu$ up to $\trans$} iff whenever $X(0)$ is distributed according to $\mu$ then $(X(t))_{t\in[0,T]}$ and the time-reversed process $(\trans(X(T-t)))_{t\in[0,T]}$ have the same  distribution (on the space of continuous paths), for every fixed $T>0$. 

Let $\cL$ be the generator of $X_t$. If $\trans$ is smooth enough -- in particular, if it is such that the composition $f\circ \trans$ belongs to the domain of $\cL$ whenever $f$ does, which we assume from now on -- one can see that $X_t$  is generalised reversible up to $\trans$ with respect to $\mu$ iff
\begin{equation}
	\label{eq: generalized revseribility}
	\int f(x)(\mathcal{L} g)(x)\,d\mu(x)=\int (g\circ \trans)(x)[\mathcal{L}(f\circ \trans)](x)\,d\mu(x), \quad \forall\,  f,g \in \cD_2(\cL)\,.
\end{equation}
From the above it is clear that a generalised-reversible process  is non-reversible in the sense that the detailed balance equation \eqref{detbalance} does not hold. 
However when $\trans$ is the identity, the above condition simply reduces to detailed balance. 
It is easy to see that if  $\cL$ is reversible up to $\trans$ with respect to $\mu$ \footnote{By this we mean that the process with generator $\cL$ is reversible up to $\trans$ with respect to $\mu$. } then $\mu$ is an invariant measure for $X_t$ (just take $f \equiv 1$ in the above and use the fact that  constant functions belong to the kernel of $\cL$). 
Note that \eqref{eq: generalized revseribility} can be reformulated as 
\begin{equation}
	\label{eq: generalized reversibility 2}
	(\cL^*f)(x)={\trans}_\#[\cL ({\trans_\#}f)](x) = (\cL ({\trans_\#}f))(\trans(x)), \quad \mbox{ where } {\trans}_\# f:=f\circ \trans \,. 
\end{equation}
Indeed,  by a change of variable, we have
\begin{align*}
	\int g(x){\trans_\#}[\cL ({\trans_\#}f)](x)\,d\mu(x)
	&=\int g(x)[\cL (f\circ \trans)(\trans (x))] d \mu(x)
	\\& \stackrel{\eqref{invarS}}{=}\int g(\trans (x))[\cL (f\circ \trans)](x)\,d\mu(x) \,
	\\&\stackrel{\eqref{eq: generalized revseribility}}{=}\int f(x)(\cL g)(x)d\mu(x).
\end{align*}

A typical example of generalised reversibility is the Langevin dynamics $(x_t, v_t)$ of equation \eqref{Langevin}, 
which is generalised reversible with respect to the measure $\mu$ in \eqref{target}  up to the velocity flip $\trans:(x,v)\mapsto (x,-v)$, see \cite{Lelievre2010}. 

\subsection{Feng-Kurtz approach to large deviation principle of empirical measures}\label{subsec:4.2}
\newcommand{\cDb}{\cD_b}
The functional setting of this subsection and the next is slightly different from the one we used in previous sections; in particular throughout Section \ref{sec:4} we drop our standing assumptions   Hypothesis \ref{SA}. 

Let $t \in [0,T]$ and  $\{X_t\nup\}_{\{ n \in \N \}}$ be a sequence of time-homogeneous independent Markov processes, each of them with common state space $\R^d$  and with common generator  $\mathcal{L}: \cDb(\mathcal{L}) \rightarrow C_b(\R^d)$, where the domain $\cDb$ of $\cL$ is the set of functions $f:\R^d \rightarrow \R$ such that $\cL f \in C_b(\R^d)$ \footnote{Let us recall that there are in general two approaches to defining the domain of a Markov operator: either one defines it as the largest set where the operator can be seen as the generator of the associated Markov semigroup, and this is substantially what we did in previous sections, or one defines it by fixing the image of $\cL$, which is what we do here. We use the space of continuous and bounded functions here because it generates the narrow (weak) topology in the space of probability measures, which is a natural topology for large deviation results in this section~ \cite{DawsonGartner1987,FengKurtz}.} and $\cL$ is a linear operator.  We assume that $\cDb$ is large enough that if \eqref{eq: generalized revseribility} holds for every $f,g, \in \cDb$ then it also holds for every $f,g \in \cD_2$.

The empirical process $\rho\nup$ associated to $\{X_t\nup\}_n$ is defined by
\begin{equation}
	\label{eq: empirical process}
	\rho^{(n)}: t\mapsto \frac{1}{n}\sum_{i=1}^n\delta_{X^{(i)}(t)},\quad t \in [0,T] \, ,
\end{equation}
where $\delta$ is the dirac delta. Denoting by $\Prd$ the space of probability measures on $\R^d$, for each $t \in [0,T]$ and $n\in \N$ fixed, $\rho_t\nup$ is a random probability measure on $\R^d$. Hence 
$\{\rho^{(n)}(\cdot)\}_{n}$ can be viewed as a sequence of $D([0,T];\mathcal{P}(\R^d))$-valued random variables where $D([0,T];\mathcal{P}(\R^d))$ denotes the Skorohod space of paths from $[0,T]$ to $\R^d$. Suppose that $\rho^{(n)}(0)$, as a sequence of $\mathcal{P}(\R^d)$- valued random variables, satisfies a LDP with a good rate function $I_0: \mathcal{P}(\R^d)\rightarrow [0,\infty)$ (see Appendix \ref{appendixA} for basic definitions about LDPs). 
Then, under rather general conditions (see \cite[Chapter 13]{FengKurtz} for non-degenerate diffusion processes, \cite{Budhiraja2012, DuongPeletierZimmer2013} for degenerate diffusion processes, \cite{Feng1994, MielkePeletierRenger2014} for finite-state continuous time Markov chains or also \cite{DawsonGartner1987}), the empirical process satisfies a large deviation principle in $D([0,T];\mathcal{P}(\R^d))$ of the form
\be\label{largedevpath}
\mathrm{Prob}\Big((\rho^{(n)}_t)_{t=0}^T\approx \pathrho\Big)\overset{n\rightarrow \infty}{\sim} e^{-n \cI_T(\pathrho)},\quad \cI_T(\pathrho):= I_0(\rho_0)+\int_0^T \L(\rho_t;\dot{\rho}_t)\,dt, \footnote{The complete explanation of this popular short-hand notation is contained in Appendix \ref{appendixA}.}
\ee
for some function $\L$ which is commonly referred to as the {\em Lagrangian}, for reasons that will be clear in few lines. In the above and throughout  the notation $\pathrho$ is to specify that we are referring to the whole path $\{\rho_t\}_{t \in [0,T]} \subset \mathcal P(\R^d)$. If $\rho\nup_{\cdot}$ satisfies the LDP \eqref{largedevpath} then  the functional $\L$ is convex in the second variable  and 
$$
\L(\rho_t; \dot\rho_t) = 0 \Leftrightarrow \dot \rho_t = \cL'\rho_t \, 
$$
where the Fokker Planck-equation on the RHS of the above is to be understood on the space of probability measures \cite[Proposition 3.1]{MielkePeletierRenger2014}. 

\noindent
Feng and Kurtz \cite{FengKurtz}  provide a formal method to calculate the large-deviation rate functional $\L$ in a variety of scenarios; in our setting, i.e. in the case in which we consider LDPs for empirical measures of i.i.d processes, such a method proceeds as follows: starting from the generator $\cL$,  we compute the so-called {\em Hamiltonian} $\mathscr{H} = \mathscr{H} (\rho;\xi) = \mathscr{H}_{\rho}(\xi)$, as follows
\begin{equation}
	\label{eq: H}
	\mathscr{H}(\rho;\xi):=\int_{\R^d} e^{-\xi(x)} \left(\mathcal{L} e^{\xi}\right)(x) \, d\rho(x) \,.
\end{equation}
For every $\rho \in \mathcal{P}(\R^d)$, we regard $\mathscr{H}$ as a function $\mathscr{H}(\rho; \cdot): \cDb(\mathscr{H}_{\rho})  \rightarrow \mathbb{R}$, where $\cDb (\mathscr{H}_{\rho})$ consists of the set of functions $\xi: \R^d \rightarrow \R$ such  that $e^{\xi} \in \cDb(\cL)$ and the right hand side of the above makes sense.    Once the Hamiltonian is obtained, if $\rho\nup_{\cdot}$ satisfies a LDP,  the Lagrangian    $\L(\rho; \cdot): \cDb(\L_{\rho})  \rightarrow \R$ is found as the Legendre transform of  $\mathscr{H}$, namely
\begin{equation}
	\label{eq: L and H}
	\L(\rho; g)=\sup_{\xi \in \cD_b(\mathscr{H}_{\rho})}\Big\{( \xi, g) - \mathscr{H}(\rho;\xi)\Big\} \,,
\end{equation}
where in the above $( \cdot, \cdot )$ is a formal $L^2$ scalar product. 

\begin{Note}\label{note:what Lagrangian be on manifolds} 
 We emphasize that  the approach of Feng and Kurtz gives a formal procedure to find the large deviation rate functional $\L$, however to make this  procedure rigorous one needs to verify many technical conditions;  \cite{FengKurtz} presents rigorous proofs of LDPs for many Markov processes. Moreover, as one would expect, specifying the appropriate domains of $\mathscr{H}$ and $\L$  is better done on a case by case basis, see e.g. \cite[Section 4.2]{MielkePeletierRenger2014}. Again to compare with the presentation of this theory on manifolds, if $\cZ$ is  a  manifold then, for every $z \in \cZ$,  $\mathscr{H}(z; \cdot): T^*_z\cZ \rightarrow \R$ and 
	$\L(z;\cdot ):T_z\cZ \rightarrow \R $. 
\end{Note}

Assuming that for every $n \in \N$  the process $\{X_t\nup\}_{t\geq 0}$ admits a unique invariant measure $\mu$ (which is the same for every $n \in \N$) and that the associated empirical process $\rho\nup$ satisfies a LDP of the form \eqref{largedevpath} with {L}agrangian $\L$,  one of the main results of  \cite{MielkePeletierRenger2014} can be informally stated as follows (see \cite[Theorem 3.3]{MielkePeletierRenger2014} for a precise statement): the process $X_t\nup$ is reversible with respect to  $\mu$ (in the sense that $\mu$ satisfies detailed balance \eqref{detbalance}) {\em if and only if } the Lagrangian $\L$ appearing in the LDP satisfies the following relation
\begin{equation}
	\label{eq: relationLdetbalance}
	\L(\rho;g)-\L(\rho;-g)= ( \vd\cS_\mu(\rho),g), \quad \mbox{for every } \rho \in \mathcal P(\R^d) , g \in \cD_b(\L_{\rho}), 
\end{equation}
where $\cS_\mu(\rho)$ has been defined in \eqref{hypto GEN definitions}. In the next subsection we show an analogous result for the case in which the invariant measure $\mu$ is only reversible up to some transformation $\trans$.

\subsection{Reversibility, generalised reversibility and large deviation principles}\label{subsec:4.3}
We keep working in the setting of Subsection \ref{subsec:4.2}, i.e. $\{X_t\nup\}_{t\geq 0}$ is a sequence of independent time-homogeneous Markov processes with common generator $\cL$; again suppose that for each $n \in \N$ the process $X^{(n)}_t$ admits an invariant measure $\mu$ which is the same for every {$n$}.
Consider the map $\varphi: D([0,T];\mathcal{P}(\R^d))\rightarrow D([0,T];\mathcal{P}(\R^d))$ defined as
\be\label{mapvarphi}
\varphi(\{\rho_t\}_{t \in [0,T]})=\{\trans_\#\rho_{T-t}\}_{t\in[0,T]} \,,
\ee
where we recall $\trans_\#\rho$ denotes the push-forward of the probability measure $\rho$ under a map $\trans: \R^d \rightarrow \R^d$, i.e.  {$\trans_\#\rho(B)=\rho(\trans^{-1}(B))$} for every Borel set $B\subset \R^d$. We make the following assumption. 
\begin{assumption} \label{assF}
	Let $\trans$ be an involutive,  volume preserving transformation  (i.e. $|J\trans|=1$)\footnote{Here $|J\trans|$ denotes the determinant of the Jacobian of $\trans$.} on $\R^d$  leaving $\mu$ invariant and satisfying the following regularity assumptions: i) $\trans$ is smooth enough so that the map $\varphi$ in \eqref{mapvarphi} is well defined (i.e. $\{\trans_\#\rho_{T-t}\}_{t\in[0,T]}$ belongs to $D([0,T];\mathcal{P}(\R^d))$ if the path $\{\rho_t\}_{t \in [0,T]}$ does) and  continuous; ii) if  $ g \in \cDb(\L)$ ($ \cDb(\mathscr{H}),$ respectively) then $g \circ \trans \in \cDb(\L)$ ($\cDb(\mathscr{H})$, respectively). 
\end{assumption}


\begin{proposition} \label{prop:irrevlagrangian} 
	With the setting and notation introduced above, suppose that $\{X\nup(0)\}_{n \in \N}$ are i.i.d random variables, distributed according to $\mu$; suppose  the empirical process $\rho^{(n)}$ associated with the sequence $\{X\nup(t)\}_{n \in N}$ satisfies a large deviation principle on $D([0,T];\mathcal{P}(\R^d))$ with a good rate function $\cI_T: D([0,T];\mathcal{P}(\R^d))\rightarrow [0,\infty)$ given by
	\begin{equation}
		\label{eq: rate functional}
		\cI_T(\pathrho)=I_0(\rho_0)+\int_0^T\L(\rho_t;\dot{\rho}_t)\,dt,
	\end{equation}
	where $I_0$ is the large deviation rate functional of $\{\rho^{(n)}(0)\}_n$,  the empirical process at the initial time.  
	Let  $\trans: \R^d \rightarrow \R^d$ be a transformation on $\R^d$ satisfying Assumption \ref{assF}.      If for each $n \in \N$ the process $\{X\nup(t)\}_t$ is reversible with respect to $\mu$ up to $\trans$ then the Lagrangian $\L$ appearing in the LDP satisfies the following relation: 
	\begin{equation}
		\label{eq: relation L}
		\L(\rho;\s)-\L(\trans_\#\rho;-\s\circ \trans)= ( \vd\cS_\mu(\rho),\s), \quad \mbox{for every } \rho \in \mathcal P(\R^d), \s \in \cDb(\L_{\rho}), 
	\end{equation}
	where $\cS_\mu(\rho)$ is as in \eqref{hypto GEN definitions}.    Moreover, if $\vd\cS_\mu(\rho) \in \cDb(\mathscr{H}_{\rho})$,   \eqref{eq: relation L} is equivalent to the following relation (expressed in terms of the Hamiltonian)
	\begin{equation}
		\label{eq: relation H}
		\mathscr{H}(\rho;\vd\cS_\mu(\rho)+\xi)=\mathscr{H}(\trans_\#\rho;-\xi\circ \trans).
	\end{equation}
	If, additionally, the Lagrangian $\L$ is twice differentiable near $(\mu;0)$, then the converse statement holds as well, i.e. \eqref{eq: relation L} implies \eqref{eq: generalized revseribility}. 
\end{proposition}

\begin{proof}
	We build on the proof of \cite[Theorem 3.3]{MielkePeletierRenger2014}. Consider the map $\varphi: D([0,T];\mathcal{P}(\R^d))\rightarrow D([0,T];\mathcal{P}(\R^d))$ defined in \eqref{mapvarphi}. 
	By the involutivity of $\trans$, $\varphi$ is invertible and involutive, and indeed the inverse map is given by
	$$
	\varphi^{-1}(\{\rho_{t}\}_{t \in [0,T]})=\{\trans_\#\rho_{T-t}\}_{t \in [0,T]} \,, \quad \mbox{for any }  \{\rho_{t}\}_{t \in [0,T]} \in D([0,T]; \mathcal P(\R^d)) \,.
	$$
	Let $\tilde{\rho}^{(n)}$ be the empirical process associated with the process $\trans(X\nup(T-t))$, i.e.
	\[
	\tilde{\rho}^{(n)}\colon t\mapsto\frac{1}{n}\sum_{i=1}^n\delta_{\trans(X^{(i)}(T-t))}.
	\]
	Using the identity
	$$
	\int_{\R^d} f(y) (\trans_\#\rho)(dy)=\int_{\R^d} f(\trans(x)) \rho(dx), 
	$$
	which holds for every continuous and bounded $f:\R^d \rightarrow \R$, 
	one can see that
	$\tilde{\rho}^{(n)}_{\cdot}=\varphi(\rho^{(n)}_{\cdot})$. 
	Since $\rho^{(n)}_{\cdot}$ satisfies a LDP with the rate functional $\cI_T$, by the contraction principle \cite[Lemma 3.11]{FengKurtz}, \footnote{It is in applying the contraction principle that we use the assumption that $\cI$ is a good rate function.} also $\tilde{\rho}^{(n)}$ satisfies a LDP with the rate functional $\tilde{\cI}: D([0,T];\mathcal{P}(\R^d))\rightarrow [0,\infty)$ given by 
	\be\label{iitilde}
	\tilde{\cI}_T(\{\rho_t\}_{t \in [0,T]})=\cI(\varphi(\{\rho_t\}_{t \in [0,T]})). \,\footnote{Since $\varphi$ is invertible, in the application of contraction principle we do not need to take an infimum over the pre-image set}
	\ee
	Letting $\widehat{\{\rho_t\}_t}=\varphi(\{\rho_t\}_t)$, the above can be rewritten as   
	\begin{equation*}
		\tilde{\cI}_T(\{\rho_t\}_t)=\cI_T(\widehat{\{\rho_t\}}_t)=I_0(\hat{\rho}_0)+\int_0^T \L(\hat{\rho}_t,\dot{\hat{\rho}}_t)\,dt.
	\end{equation*}
	We claim that the following identity holds:
	\begin{equation}\label{minus}
		-\dot{\hat{\rho}}_t(x)=(\dot{\rho}_{T-t}\circ \trans)(x) \,, \quad \mbox{for every }  x \in \R^d\,.
	\end{equation} 
	We will prove this claim below and for the time being we move on with the main argument.
	
	\noindent
	Note now that the generalised reversibility of the underlying particle systems implies that $\cI=\tilde{\cI}$. This is because, by generalised reversibility, for each $n$, $(X\nup(t))_{t\in[0,T]}$ and the time-reversed process $(\trans(X\nup(T-t)))_{t\in[0,T]}$ have the same  distribution, for every fixed $T>0$. Hence, 
	$$
	\mathrm{Prob} \left( \{\rho\nup_t\}_{t \in [0,T]} \in \Theta\right) = 
	\mathrm{Prob} \left( \{\rho\nup_t\}_{t \in [0,T]} \in \varphi(\Theta)\right)\,,
	$$
	{for every } open set $\Theta \subseteq D([0,T]; \mathcal P(\R^d)) $.
	From  $\cI=\tilde{\cI}$ and  \eqref{iitilde}, 
	for any curve $\{\rho_t\}_ t  \subset D([0,T]; \mathcal P(\R^d)) $ we then have 
	\begin{align*}
		I_0(\rho_0)+\int_0^T\L(\rho_t;\dot{\rho}_t)\, dt&=I_0(\hat{\rho}_0)+\int_0^T \L(\hat{\rho}_t;\dot{\hat{\rho}}_t)\,dt
		\\& \stackrel{\eqref{minus}}{=}I_0(\trans_{\#}\rho_T)+\int_0^T\L(\trans_{\#}\rho_{T-t};-\dot{\rho}_{T-t}\circ \trans)\, dt
		\\&=I_0(\trans_{\#}\rho_T)+\int_0^T\L(\trans_{\#}\rho_t;-\dot{\rho}_t\circ \trans)\, dt \,.
	\end{align*}
	Since $\{X\nup(0)\}_n$ are i.i.d. with distribution $\mu$, according to Sanov's theorem \cite[Theorem 6.2.10]{DemboZeitouni-book} we have $I_0(\rho_0)=\cS_{\mu}(\rho_0)$. Thus, 
	\begin{equation}
		\cS_{\mu}(\trans_{\#}\rho_T)-\cS_{\mu}(\rho_0)=\int_0^T \left[ \L(\rho_t;\dot{\rho}_t)-\L(\trans_{\#}\rho_t;-\dot{\rho}_t\circ \trans) \right]\, dt
	\end{equation}
	Differentiating the above with respect to $T$, we then get
	\[
	( \vd \cS_{\mu}(\trans_{\#}\rho_T),\dot{\rho}_T\circ \trans)=\L(\rho_T;\dot{\rho}_T)-\L(\trans_{\#}\rho_T; -\dot{\rho}_T\circ \trans) \,.
	\]
	Since $\{\rho_t\}_t$ is an arbitrary curve in $D([0,T]; \mathcal P(\R^d))$,   $\rho_T$  is an arbitrary probability measure, as is arbitrary the element $\dot{\rho}_T \in \cDb(\L_{\rho_T})$,   hence  
	\[
	( \vd \cS_{\mu}(\rho),\s\circ \trans)=\L(\trans_{\#}\rho;\s)-\L(\rho;-\s\circ \trans),
	\]
	for any $\rho,\s$ in the appropriate spaces, 
	which is equivalent to \eqref{eq: relation L}. To conclude the proof of \eqref{eq: relation L} we just need to prove the claim \eqref{minus}. 
	This can be seen from the following calculations
	\begin{align*}
		\int \dot{\hat{\rho}}_t (x) f(x)\,dx
		&=   \frac{d}{dt}\int \hat{\rho}_t(x) f(x)\,dx
		\\
		&=-\int \dot{\rho}_{T-t}(x) (f\circ \trans)(x)\,dx=-\int (\dot{\rho}_{T-t}\circ \trans)(x) f(x)\,dx ,
	\end{align*}
	where the second equality follows from the definition of $\hat{\rho}_t$ and having used the fact that $\trans$ is volume preserving. 
	The convexity of $\L$ (implied by the LDP) now allows us to find the relation \eqref{eq: relation H}  as follows:  
	\begin{align*}
		\mathscr{H}(\rho; \vd\cS_\mu(\rho)+\xi)&=\sup_s\{( \vd\cS_\mu(\rho)+\xi,\s )-\L(\rho;\s)\}
		\\& \stackrel{\eqref{eq: relation L}}{=}\sup_\s\{( \xi,\s )-\L(\trans_\#\rho;-\s\circ \trans)\}
		\\&=\sup_\s\{( \xi,-\s\circ \trans )-\L(\trans_\#\rho;\s)\}
		\\&=\sup_\s\{( -\xi\circ \trans,  \s )-\L(\trans_\#\rho; \s)\}
		\\&=\mathscr{H}(\trans_\#\rho; -\xi\circ \trans)\,.
	\end{align*}
	Let us now prove the converse implication, i.e. that \eqref{eq: relation L} implies \eqref{eq: generalized revseribility}. This can be done by a modification of the proof of  \cite[Theorem 3.3]{MielkePeletierRenger2014}, so we don't repeat all the details but we just explain the main steps. Standing the observations of  \cite[Remark 3.4]{MielkePeletierRenger2014}, let us define the function
	$$
	a(\epsilon_1, \epsilon_2, \dots, \epsilon_6) = 
	\L((1+ \epsilon_1 \rho + \epsilon_4 \rho_\#)\mu ; \, 
	\cL'((1+\epsilon_2\rho+\epsilon_3 g+\epsilon_5 \rho_\#+\epsilon_6 g_\#)\mu)) \,,
	$$
	where $\epsilon_j \in \R$ for every $j \in\{1, \dots, 6\}$ and, just for the rest of this proof,  we set $g_\#=\trans_\#g$, similarly for $\rho_\#$.  
	Since $\vd^2 \cS_{\mu}(\mu)[g_1, g_2]=
	\int_{\R^d} \frac{dg_1}{d\mu}\frac{dg_2}{d\mu} d\mu $, the statement that we want to prove can be rewritten as 
	\be\label{modstat}
	\vd^2 \cS_{\mu}(\mu)[ \mu \rho, \cL'(\mu g)] =  
	\vd^2 \cS_{\mu}(\mu)[\mu g_\# , \cL' (\mu\rho_\# )] \,.
	\ee
	To prove the above we make two observations: firstly, by \eqref{eq: relation L} (applied to the functions $(1+\epsilon_1\rho)\mu$ and $\cL'\big((1+\epsilon_3 g)\mu\big)$), we have 
	\be\label{ob1}
	a(\epsilon_1, 0, \epsilon_3, 0,0,0) -
	a(0,0,0,\epsilon_1, 0, -\epsilon_3)= 
	(\vd \cS_{\mu}((1+\epsilon_1\rho)\mu)), \cL'(\epsilon_3 g \mu))\, ; 
	\ee
	secondly, by the validity of the LDP, the function 
	$\L((1+\epsilon_1 \rho+\epsilon_4 \rho_\#)\mu ; \cL'[(1+\epsilon_1 \rho+\epsilon_4 \rho_\#)\mu] +h )$ attains its minimum when $h=0$, hence
	\be\label{ob2}
	D_3 a(\epsilon_1, \epsilon_1, 0,\epsilon_4, \epsilon_4, 0)=0 \,\,\mbox{ and } \,\,
	D_6 a(\epsilon_1, \epsilon_1, 0,\epsilon_4, \epsilon_4, 0)=0 \,.
	\ee
	If we differentiate \eqref{ob1} with respect to $\epsilon_3$ and then with respect to $\epsilon_1$ (and then calculate both derivatives in $\epsilon_1=\epsilon_3=0$), we get 
	$$
	D_{13}a(0,\dots, 0)+D_{46}a(0,\dots, 0) =
	\vd^2\cS_{\mu} (\mu)[\rho\mu, \cL'(g\mu)]\,.
	$$
	From \eqref{ob2} instead, we have 
	\begin{align*}
		&D_{13}a(0,\dots, 0)+D_{23}a(0,\dots, 0)	=0\\
		&D_{46}a(0,\dots, 0)+D_{56}a(0,\dots, 0) =0 \,.
	\end{align*}
	From the above we deduce
	$$
	-D_{23}a(0,\dots, 0)-D_{56}a(0,\dots, 0)= \vd^2\cS_{\mu} (\mu)[\rho\mu, \cL'(g\mu)]\,.
	$$
	If we calculate explicitly the two derivatives on the  LHS of the above, we obtain
	$$
	\vd^2_{gg}\L(\mu,0)\left((\cL'(g\mu),\cL'(\rho\mu))+
	(\cL'(\mu\rho_\#), \cL'(\mu g_\#)) = \vd^2\cS_{\mu} (\mu)[\rho\mu, \cL'(g\mu)]\,,
	\right)
	$$
	where $\vd^2_{gg}\L$ denotes the second derivative of $\L$ with respect to its second argument. The LHS of the above expression remains unchanged upon swapping $g$ with $\rho_\#$ and $\rho$ with $g_\#$, hence \eqref{modstat}.  This concludes the proof. 
	
\end{proof}

\begin{Note}
	\label{sec: alternative computation for H}
	\begin{itemize}\
	\item 
	\textup{In Proposition \ref{prop:irrevlagrangian} we started with assuming that the process $\{X\nup(t)\}_t$ is reversible with respect to $\mu$ up to $\trans$ and proved that, if this is the case, then the Lagrangian needs to satisfy the relation  \eqref{eq: rate functional}. From the relation for the Lagrangian we then deduced property \eqref{eq: relation H} for the Hamiltonian.   
		An alternative  way of proceeding is to derive   \eqref{eq: relation H} directly  from  the generalized reversibility condition \eqref{eq: generalized revseribility} expressed in terms of the generator, starting from the explicit formula of $\mathscr{H}$ in \eqref{eq: H}:
		\begin{align*}
			\mathscr{H}(\rho,\vd\cS_\mu(\rho)+\xi)&=\mathscr{H}(\rho,\log\frac{d\rho}{d\mu}+1+\xi)
			\\&=\int_{\R^d} e^{-\log\frac{\rho(x)}{\mu(x)}-1-\xi(x)}
			\mathcal{L} \left( e^{\log\frac{\rho}{\mu}+1+\xi} \right)(x)
			\rho(x) dx   
			\\&=\int_{\R^d} e^{-\xi(x)}\left(\mathcal{L}  e^{\log\frac{\rho}{\mu}+\xi}\right)\!(x)
			\,\mu(x) dx
			\\&
			\stackrel{\eqref{invarS},
				\eqref{eq: generalized revseribility}}{=}\int_{\R^d} 
			\rho(\trans(x)) e^{(\xi \circ \trans)(x)}
			\left( \cL e^{-(\xi \circ \trans)} \right)\!(x) dx
			\\
			& = \mathscr{H}(\trans_{\#} \rho; - \xi \circ \trans)\,.
		\end{align*}
		Note that in the first equality we have used the linearity of $\mathcal{L}$ to take the constant $e^{-1}$ out of the argument of $\cL$.  
		Once \eqref{eq: relation H} is formally obtained, \eqref{eq: relation L} can be derived from \eqref{eq: relation H} as follows ({in the second equality below we use the fact that for any function $f=f(\xi)$ one has $\sup_{\xi}f(\xi) = \sup_{\xi}f(\xi+a)$, for any constant $a$})
		\begin{align*}
			\L(\rho;\s)&=\sup_{\xi}\{( \xi, \s)-\mathscr{H}(\rho;\xi)\}
			\\&=\sup_{\xi}\{( \xi+\vd\cS_\mu(\rho), \s)-\mathscr{H}(\rho{;}\xi+\vd\cS_{{\mu}}(\rho))\}
			\\& {=}( \vd\cS_{{\mu}}(\rho),\s)+\sup_{\xi}\{(\xi,\s)-
			\mathscr{H}(\rho;\xi+\vd\cS_\mu(\rho))\}
			\\&\stackrel{\eqref{eq: relation H}}{=}( \vd\cS_\mu(\rho),\s)
			+\sup_{\xi}\{(\xi,\s)-\mathscr{H}
			(\rho\circ \trans;-\xi\circ \trans)\}
			\\& {=}( \vd\cS_\mu(\rho),\s)
			+\sup_{\xi}\{(\xi\circ \trans,\s)-\mathscr{H}(\rho\circ \trans;-\xi)\}
			\\&=( \vd\cS_\mu(\rho),\s)
			+\sup_{\xi}\{(\xi,\s\circ \trans)-\mathscr{H}(\rho\circ \trans,-\xi)\}
			\\&=( \vd\cS_\mu(\rho),\s)+\sup_{\xi}\{( -\xi,-\s\circ \trans)-\mathscr{H}(\rho\circ \trans;-\xi)\}\\&=
			( \vd\cS_\mu(\rho),\s)
			+\sup_{\xi}\{( \xi,-\s\circ \trans)-\mathscr{H}(\rho\circ \trans;\xi)\}	
			\\&=( \vd\cS_\mu(\rho),\s)+\L(\rho\circ \trans ; -\s\circ \trans).
		\end{align*}
		The observations in \cite[Remark 3.5]{MielkePeletierRenger2014} can be seen as a specific instance of the above calculations, in the case in which $\trans$ is the identity. 
		}
\item {By combining Theorem \ref{theorem:decomposition} and Proposition \ref{prop:irrevlagrangian}, one can derive a gradient flow structure of the Fokker Planck equation $\partial_t\rho=\mathcal{L}'\rho$ from a large deviation rate functional by computing the Hamiltonian $\mathscr{H}$ from the Lagrangian $\mathscr{L}$ appearing in the rate functional \eqref{eq: rate functional} then the dissipation potential $\Psi^\star$ from $\mathscr{H}$ using \eqref{eq: Psi from Ls} (the entropy functional being the relative entropy). As shown in \cite{MielkePeletierRenger2014} for a diffusion process with detailed balance, the resulting gradient flow structure is precisely the (quadratic) Wasserstein gradient flow discovered by Jordan-Kinderlehrer-Otto \cite{JKO98,AGS} (cf. \eqref{genWasgradflow}); while for a continuous time Markov chain with detailed balance, it is a non-quadratic gradient flow (cf. \eqref{non-quadratic gradient flow}), corresponding to the dissipation potential given by \cite[Eq. (4.6)]{MielkePeletierRenger2014}
$$
\Psi^\star(\rho,\xi)=\frac{1}{2}\sum_{i,j=1}^J \sqrt{\rho_i\rho_j Q_{ij}Q_{ji}}(\cosh(\xi_j-\xi_i)-1),
$$
where $Q=(Q_{ij})\in \mathbb{R}^{J\times J}$ denotes the generator of the chain. This is not the gradient flow structure for a (detailed balance) Markov chain found by Maas, Chow et.al. and Mielke \cite{Maas2011,Chow, Mielke2011b} (although the entropy functional being the same relative entropy, the dissipation potential in these papers is quadratic), see also the Introduction for relevant discussion. In Section 6, we derive pre-GENERIC structures for non-reversible processes and PDMP processes using \eqref{eq: H} to directly compute the Hamiltonian from the generator.
		}
\end{itemize}
\end{Note}


\section{Examples: Diffusion Processes and PDMPs}\label{sec:sec6}

In this section we consider two classes of processes, Diffusion Processes and Piecewise Deterministic Markov Processes (PDMPs), Subsection \ref{Diffusion processes} and Subsection \ref{sec:PDMPs}, respectively. We will show how the theory developed in previous sections can be applied to such classes of Markov evolutions and indicate assumptions under which the functional framework of previous sections can be made rigorous. 

\subsection{Diffusion processes} \label{Diffusion processes}We demonstrate how the results in Sections \ref{sec:3} and Section \ref{section:decomposition} can be applied to diffusion processes. Let $b\in C^1(\R^d;\R^d)$ be globally Lipschitz and $\sigma\in C^2(\R^{d};\R^{d\times m})$ be uniformly elliptic. We consider the following Ito diffusion process
\begin{equation}
	\label{eq: diffusion}
	dX_t=b(X_t)\,dt+\sqrt{2}\sigma(X_t)\,d\beta_t, \quad X_0=x \in \R^d\, ,
\end{equation}
where $(\beta_t)_{t\geq 0}$ is a standard Brownian motion in $\R^m$.
The generator of this process is given, on suitably smooth functions $u: \R^d \rightarrow \R$,  by
\begin{equation}
	\label{eq: diffusion operator}
	(\mathcal{L}u)(x)=b(x)\cdot\nabla u(x)+D(x):\nabla^2 u(x),
\end{equation}
where $D(x)=\sigma(x)\sigma(x)^\top$ and the notation ``:" denotes the Frobenius inner product of two matrices. The Fokker Planck equation associated to \eqref{eq: diffusion} is
\begin{equation}
	\label{eq: FPE}
	(\partial_t\rho_t)(x)=(\mathcal{L}'\rho_t)(x)=\div(D(x)\nabla \rho_t(x)-b(x)\rho_t(x)).
\end{equation}
Assume that the process admits an invariant measure  $\mu(dx)=\mu(x)\,dx$ (simple sufficient conditions for this can be found e.g. in \cite{Pardoux}), which is necessarily unique and smooth in view of our ellipticity assumption. We can also assume that $\mu$ is positive everywhere. By definition, such an invariant measure satisfies 
\begin{equation}
	\label{eq: invariant FPE}
	\mathcal{L}'\mu=\div(D\nabla\mu-b\mu)=0,
\end{equation}
With this setup we can take $\dense$ to be the set of (positive) Schwartz functions. 
\subsubsection{Wasserstein pre-GENERIC and hypocoercive formulations}
Let $h_t(x)=\rho_t(x)/\mu(x)$. Then $h$ satisfies the hypocoercive modified-Kolmogorov equation \eqref{m-KEactual} {where (see \cite[Proposition 3]{Villani})\footnote{See the discussion after \eqref{m-KElinearhypocform} on the notation $A^*A$.}
	$$
	B:=(b-D\nabla\log\mu)\cdot\nabla,\quad A:=\sigma\nabla \quad\text{and}\quad A^*g=-\mathrm{div}(\sigma^{{\top}} g)-(\nabla \log \mu, \sigma^{{\top}} g).
	$$
}
Since $B$ and $A$ are derivations on $\R^d$, they satisfy both chain and product rules. Hence according to Section \ref{linhypoc to quadpre-Gen}, the Fokker-Planck equation \eqref{eq: FPE} can be written in the quadratic pre-GENERIC form \eqref{quadratic pre-generic} with the entropy functional given by the relative entropy of $\rho$ with respect to $\mu$ (see \eqref{hypto GEN definitions}) and (see \eqref{hypto GEN definitions2})
\begin{align}
	&W\rho=B'\rho=\div\Big(\rho D\nabla\log(\mu)-b\rho\Big),\label{eq: W-FPE}
	\\& M_{\rho}(\xi)=2A'(\rho A\xi)=-2\div\Big(\rho D\nabla \xi\Big).    \label{eq: M-FPE}
\end{align}

\subsubsection{Pre-GENERIC formulation of the FPE from the symmetric-antisymmetric decomposition}

In this section, we will cast the Fokker-Planck equation \eqref{eq: FPE} into the pre-GENERIC formulation \eqref{splitting2} using the symmetric-antisymmetric decomposition of the generator $\cL$ (see \eqref{eq: diffusion operator}) and show that this is the same as the quadratic pre-GENERIC structure \eqref{eq: W-FPE}-\eqref{eq: M-FPE} obtained from the hypocoercive modified-Kolmogorov equation \eqref{m-KEactual} in the previous section.

First, we find the symmetric and antisymmetric parts in $L^2_\mu$ of $\cL$, using the relation
$$
\cL_s\phi=\frac{\cL\phi+\cL^*\phi}{2}=\frac{\cL\phi+ \mu^{-1}\cL'(\mu\phi) }{2},
$$
{where the second equality in the above follows from the following calculation
	$$
	(\varphi, \mathcal{L}^*\phi)=\langle \mu^{-1}\varphi, \mathcal{L}^*\phi\rangle=\langle \phi,\mathcal{L}(\mu^{-1}\varphi)\rangle=(\varphi,\mu^{-1}\mathcal{L}'(\mu\phi)).
	$$}
The symmetric part $\cL_s$ can be found explicitly
\begin{align*}
	\cL_s\phi&=\frac{1}{2}\Big(\cL\phi+ \mu^{-1}\cL'(\mu\phi)\Big) 
	\\&=\div(D\nabla\phi)+D\nabla\phi\cdot\nabla\log\mu \,,
\end{align*}
 we have used standard calculations and  the stationary condition \eqref{eq: invariant FPE} of $\mu$ to obtain the second equality. The anti-symmetric part is then given by
$$
\cL_a\phi=\cL\phi-\cL_s\phi=b\cdot\nabla\phi-D\nabla\phi\cdot\nabla\log\mu.
$$
The $L^2$-dual operator $\cL_s'$ of $\cL_s$ is 
\begin{equation}
	\label{eq: Ls dual}
	\cL'_s\rho=\div(D\nabla\rho)-\div(\rho D\nabla\log\mu)=\div\left[\rho D\nabla\big(\log(\rho/\mu)\big)\right]=\div\left[\rho D\nabla \big(\vd\cS_\mu(\rho)\big)\right].
\end{equation}
Next, we find the Hamiltonian $\mathscr{H}$ and its symmetric part $\mathscr{H}_s$ from $\cL$ and $\cL_s$ (see \eqref{eq: decompose H} and \eqref{HaHs}). 
Using \eqref{eq: diffusion-carreduchamp} we obtain
\begin{equation}
	\mathscr{H}(\rho;\xi)=\int e^{-\xi}\cL e^{\xi}\rho=\int e^{-\xi}\Big[e^{\xi}(\cL \xi+\Gamma(\xi,\xi))\Big]\rho=(\xi, \cL'\rho)+( D\nabla\xi\cdot\nabla\xi,\rho). \label{eq: Hdiffusion}
\end{equation}
Similarly we get
$$
\mathscr{H}_s(\rho;\xi)=\int e^{-\xi}\cL_s e^{\xi}\rho=(\xi, \cL_s'\rho)+(D\nabla\xi\cdot\nabla\xi,\rho).
$$
Having obtained $\mathscr{H}_s$, we now calculate $\psi^\star$ using \eqref{eq: Psi from Ls}
\begin{align*}
	\psi^\star(\rho;\xi)&=\mathscr{H}_s\left(\rho;\xi+\frac{1}{2}\vd\cS_\mu(\rho)\right)-\mathscr{H}_s\left(\rho;\frac{1}{2}\vd\cS_\mu(\rho)\right)
	\\&=(\xi,\cL_s'\rho)+\left(D\nabla (\vd\cS_\mu(\rho))\cdot\nabla\xi,\rho\right)+( D\nabla\xi\cdot\nabla\xi,\rho)
	\\&=(\xi,\cL_s'\rho)- \Big(\xi,\div\big[\rho D\nabla (\vd\cS_\mu(\rho))\big]\Big)+( D\nabla\xi\cdot\nabla\xi,\rho)
	\\&=(D\nabla\xi\cdot\nabla\xi,\rho),
\end{align*}
where we have used \eqref{eq: Ls dual} to obtain the last equality. The Frechet derivative of $\psi^\star$ with respect to its second argument is given by
$$
\vd_\xi\psi^\star(\rho;\xi)=-2\div(\rho D\nabla \xi),
$$
which is equal to $M_\rho(\xi)$ in \eqref{eq: M-FPE}. Finally, we have
\begin{align*}
	W\rho=\cL_a' \rho=\div(\rho D\nabla\log\mu-b\rho),
\end{align*}
which is the same as \eqref{eq: W-FPE}. Hence, the pre-GENERIC formulation \eqref{splitting2} of the Fokker-Planck equation obtained from the symmetric-antisymmetric decomposition of the generator is the same as the quadratic pre-GENERIC structure \eqref{eq: W-FPE}-\eqref{eq: M-FPE} obtained from the hypocoercive modified-Kolmogorov equation \eqref{m-KEactual}. 
\subsubsection{Kinetic Fokker Planck equation}
We now consider the kinetic Fokker Planck equation mentioned in the Introduction demonstrating Proposition \ref{prop:irrevlagrangian}. This equation is not elliptic but we can still take $\dense$ to be the set of Schwartz functions, see \cite{Villani}. We recall that the operator $\cL$ for the kinetic Fokker Planck  equation is given by (here in multidimensional setting)
\begin{equation*}
	(\cL\xi)(x,v) =v\cdot \nabla_x\xi-\nabla V(x)\cdot\nabla_v \xi-v\cdot\nabla_v \xi+\Delta_v\xi=\mathcal{J}\nabla H\cdot\nabla\xi+\div_p\Big[\rho\nabla_p\big(\vd\cS_\mu(\rho)\big)\Big],
\end{equation*}
where $\cS_\mu$ is the relative entropy \eqref{eq: entropykFPE} and
$$
\mathcal{J}:=\begin{pmatrix}
	0&I\\ -I&0
\end{pmatrix}, \quad H(x,v):=V(x)+\frac{1}{2}|v|^2
$$
The kinetic Fokker-Planck equation is a special case of \eqref{eq: diffusion} with
$$
b(x,v)=\begin{pmatrix}
	v\\
	-\nabla V(x)-v
\end{pmatrix}, \quad \sigma=\begin{pmatrix}
	0&0\\0 &I
\end{pmatrix}.
$$
The operator $\cL$ is reversible with respect to $d\mu(x,v)=Z^{-1}e^{-(V+\frac{1}{2}|v|^2)}dxdv$ up to $\mathcal{F}(x,v)=(x,-v)$  \cite{Lelievre2010}. According to \eqref{eq: Hdiffusion}, the Hamiltonian $\mathscr{H}:\mathcal{D}(\mathscr{H})\rightarrow \mathbb{R}$ is given by 
\begin{equation}
	\label{eq: HkFPE}
	\mathscr{H}(\rho,\xi)=\int (\cL \xi+|\nabla_p\xi|^2)\rho=( \cL'\rho,\xi)+( |\nabla_p\xi|^2,\rho).
\end{equation}
We now verify directly the relation \eqref{eq: relation H} using \eqref{eq: HkFPE}. In fact, the LHS of \eqref{eq: relation H} is
\begin{align*}
	\mathscr{H}(\rho; \vd\cS_\mu(\rho)+\xi)&= (\cL'\rho, \vd\cS_\mu(\rho)+\xi)+( |\nabla_p(\vd\cS_\mu(\rho)+\xi)|^2,\rho)
	\\&=\left(-\mathcal{J}\nabla H \cdot \nabla \rho+\div_p\Big[\rho\nabla_p\big( \vd\cS_\mu(\rho)\big)\Big],\vd\cS_\mu(\rho)+\xi\right)+\left( |\nabla_p(\vd\cS_\mu(\rho)+\xi)|^2,\rho\right)
	\\&=-(\mathcal{J}\nabla H\cdot\nabla\rho,\xi)+(\nabla_p (\vd\cS_\mu(\rho))\cdot\nabla_p\xi,\rho)+(|\nabla_p\xi|^2,\rho).
\end{align*}
The RHS of \eqref{eq: relation H} is
\begin{align*}
	\mathscr{H}(\trans_\#\rho; -\xi\circ \trans)&=( \cL'(\trans_\#\rho),-\xi\circ \trans)+(|\nabla_p (-\xi\circ \trans)|^2,\trans_\#\rho)
	\\&=(\mathcal{J}\nabla H\cdot\nabla (\trans_\#\rho),\xi\circ \trans)+\Big(\nabla_p \big(\vd\cS_\mu(\trans_\#\rho)\big)\cdot\nabla_p(\xi\circ \trans\Big), \trans_\#\rho) 
	\\&\qquad+ (|\nabla_p (\xi\circ \trans)|^2,\trans_\# \rho)
	\\&=-(\mathcal{J}\nabla H\cdot\nabla\rho,\xi)+\Big(\nabla_p (\vd\cS_\mu(\rho))\cdot\nabla_p\xi,\rho\Big)+(|\nabla_p\xi|^2,\rho).
\end{align*}
Thus $ \mathscr{H}(\rho;\vd\cS(\rho)+\xi)=\mathscr{H}(\trans_\#\rho; -\xi\circ \trans)$ as expected. Finally, we recall that the GENERIC structure of Generalised Langevin equations has been studied in \cite{DuongPavl}.

\subsection{Hamiltonian -Piecewise Deterministic Markov Processes (PDMPs)}
\label{sec:PDMPs}
A PDMP $\{z_t\}_{t\geq 0}$ on some state space say $Z$ is a continuous-time stochastic process that evolves as follows: between random times (usually called random events), it evolves according to a deterministic dynamics described by an ODE; random events happen at a rate $\lambda(z)$; if before the random event the value of the process was $z_{t-}$, when the random event takes place  the process jumps to a new position selected according to a Markov kernel $Q(z_{t-}; \cdot)$ so that $z_t \sim Q(z_{t-}; \cdot)$; the new position $z_t$ is then used as the new initial condition to start again the ODE evolution. More precisely, a PDMP is described by three ingredients:
\begin{description}
	\item[i)] an ODE with drift $F: Z \rightarrow Z$,  namely
	$$
	\frac{dz_t}{dt}= F(z_t)
	$$
	where $F$ is a nice enough function to ensure at least that the Cauchy problem for this ODE  is well-posed for every initial datum in $Z$;
	\item[ii)] a rate function $\lambda: Z \rightarrow \R_+$; \footnote{Because of the way the rate function will be employed, we will have that the probability of an event taking place in the interval $[t, t+\delta]$ is precisely $\lambda(z_t) \delta + o(\delta)$.}
	\item[iii)] a family of Markov kernels $Q$ on $Z$; that is,  for every $z \in Z$, $Q(z; \cdot)$ is a Markov kernel on $Z$. 
\end{description}
With this notation in place, the generator of the dynamics is given by 
\be\label{genPDMPgen}
(\cL f)(z)= b(z) \cdot \nabla_z f(z) +
\lambda(z) \int_{Z} \left[f(z') - f(z)\right] Q(z,dz')
\ee
for sufficiently smooth functions $f:Z \rightarrow \R$. The associated Fokker Planck operator $\cL'$ is 
\be\label{gendualPDMPgen}
(\cL'f)(z)= - \nabla_z \cdot (b(z)f(z))-\lambda(z)f(z)+\int_{Z}\lambda(z')f(z')Q(z',z)dz' \,,
\ee
having assumed for simplicity that the kernel $Q(z;dz')$ has a density which, with abuse of notation, we call $Q(z,z')dz'$. 
Clear references on the matter  are \cite{FearRos, DoucetVanetti}. In particular in \cite{FearRos} it is shown how some very popular sampling algorithms are of the above form, with $Z=\{(x,v) \in\R^d \times \{-1,1\}^d\}$,  $Z=\R^d \times \{\mbox{some discrete set}\}$ or $Z=\R^d \times \R^d$. The latter case  is the one we will be considering in our example and we will denote $z=(x,v)$ a point in  $\R^d \times \R^d$. 

A classical example which belongs to this setup is the so-called {\em Andersen thermostat} (see \cite[IIA-Sec2]{kpel}) which, from a sampling perspective, is the Hybrid Monte Carlo algorithm (see \cite{RadfordNeal} and \cite{OttStuart, OttFursov} for related Hamiltonian samplers). The Andersen thermostat can be understood as follows:  suppose we want to sample from the measure $\mu$ defined in \eqref{target}, or, more precisely, from its multidimensional version, 
\be\label{invmeas multid}
\mu(x,v)= e^{-V(x)}e^{-\lv v\rv^2/2} \, ,
\ee
where the normalization constant has now been included in the potential $V:\R^d \rightarrow \R$. 
One way of doing so is to employ a deterministic Hamiltonian dynamics, i.e. a dynamics with generator 
$$
\cL_H f : = v \cdot \nabla_x f(x,v)  -   \nabla_x  V(x) \cdot\nabla_v f(x,v) \,
$$
which is just the Liouville operator associated with the Hamiltonian $H:\R^d \times \R^d \rightarrow \R$ \footnote{This is a classical finite-dimensional Hamiltonian function, i.e. here we don't use the word {\em Hamiltonian} with the same meaning as for the functional $\mathscr H$ introduced in equation  \eqref{eq: H}.}  $H(x,v) =  V(x)+ |v|^2/2$. This dynamics preserves $\mu$ but it is clearly not ergodic on $Z$; however it can be modified by resampling the momentum variable according to a standard Gaussian on $\R^d$, 
$
\pi_0 \sim \mathcal{N}(0, I_d), $ where $I_d$ is the $d$-dimensional identity matrix. 
The resulting dynamics is the Andersen thermostat, which evolves according to the following generator:
\be\label{GenAndersenThermosta}
\cL_{AT}f:=\cL_H f+ \cL_Q f
\ee
where 
$$
(\cL_Q f)(x,v): = \lambda_r \left[(\tilde Q f)(x,v)-f(x,v)\right]
$$
with the constant $\lambda_r>0$ being the rate at which the momentum variable is resampled and
$$
(\tilde Q f)(x,v):= \int_{\R^d} f(x, \xi) \, d\pi_0(\xi) \,.
$$
The dual $\cL_Q'$ of $\cL_Q$ is often referred to as the BGK collision operator, see \cite{DMS2015}; the dual of $\cL_{AT}$, $\cL_{AT}'$, is given by 
$$
\cL_{AT}' = -\cL_H + \cL_Q'\, 
$$
with $\cL_Q'$ as in Lemma \ref{lemma:dualsPDMPexample} below; such an operator  can be put in  pre-GENERIC form \eqref{pre-GENERIC} (see \cite{kpelmath, kpel}) 
by 
setting
$
W(f) = -\cL_H f \,$ and using 
the dissipation potential 
$$
\psis_f(\xi) = \frac{\lambda_r}{(2\pi)^d} \iiint dx dv dv'\left[\cosh(\xi(x,v') - \xi(x,v)) -1  \right] e^{- \frac{|v|^2+|v'|^2}{4}} \sqrt{f(x,v)f(x,v')}\,
$$
and  the relative entropy $\cS_{\mu}(f)$ in \eqref{hypto GEN definitions} as entropy functional;   for $\mu$ as in \eqref{invmeas multid}, $\cS_{\mu}(f)$ becomes 
\begin{align}
	\cS(f)&= \cS_{\mu}(f) = \iint   dx dv f(x,v) \log\left( \frac{f(x,v)}{\mu(x,v)} \right) \nonumber \\
	&\stackrel{\eqref{hypto GEN definitions} }{=} \iint f \log f dx dv+ \iint f \left( V(x) + \frac{|v|^2}{2} \right) dx dv \,. \label{entropy}
\end{align}
In particular, $-\cL_H$ corresponds to the ``non gradient-flow" part of the dynamics, while, with $\psis$ and $\cS_{\mu}$ as in the above,  $\cL_Q'$ is in gradient flow  form:
\be\label{psiforLQ}
\cL_Q'f =  \vd_\xi\psis\left(f; -\frac{\vd\cS_{\mu}(f)}{2}\right)\,.
\ee
It is straightforward to verify that the degeneracy condition \eqref{degconpre-GENERIC} is satisfied, i.e.
\be\label{degconAndersentherm}
(W(f), \vd\cS_{\mu}(f)) = (-\cL_H f, \vd\cS_{\mu} (f))=0 \,
\ee
so that the entropy \eqref{entropy} is dissipated along the Andersen thermostat  flow $\pa_t f_t = \cL_{AT}'f_t$. 
\begin{Note}\label{note:any potentialandersthermost}
	\textup{We note now in passing that if in the expression \eqref{entropy} one replaces the potential $V(x)$ with any other (suitable) function of $x$ only,  $\hat V = \hat V (x)$ \footnote{suitable in the sense that the integral in \eqref{entropy} still makes sense} then the representation \eqref{psiforLQ} still holds, and $\cL_{AT}$ can still be cast in pre-GENERIC  form \eqref{pre-GENERIC},  but the degeneracy condition \eqref{degconAndersentherm} holds iff $\nabla_x\hat V(x) = \nabla_x V(x)$,    coherently with the fact that the dynamics converges to the measure  $\mu$.  }
\end{Note}

\smallskip \noindent
{\bf Casting Hamiltonian PD-MCMC in pre-GENERIC form. } If the Hamiltonian flow is too complicated to simulate numerically, one may wish to use an auxiliary potential $\tilde{V}(x)$ and  the associated Liouville operator
$$
(\tilde\cL_H f)(x,v) : = v \cdot \nabla_x f(x,v)  -   \nabla_x  \tilde{V}(x) \cdot\nabla_v f(x,v)\,.
$$
However the generator $\tilde \cL_H+\cL_Q$ does not preserve $\mu$ (it preserves $\tilde\mu = e^{-\tilde{V}(x)}\times\pi_0(v)$ instead). To correct for this bias one can introduce a further jump process, which restores the invariance of $\mu$. The resulting dynamics is a so-called Hamiltonian PD-MCMC, see \cite[Section 2.4]{DoucetVanetti}, with generator given by
\be\label{genhampdmp}
\cL f(x,v) =(\tilde \cL_H f)(x,v) +  (\cL_Q f) (x,v) + (\cL_R f)(x,v)  
\ee
where 
$$
(\cL_Rf)(x,v):= \tilde\lambda(x,v) [ f(x,R_x^{\tilde U} v) - f(x,v)], 
$$
$\tilde U (x) = V(x) - \tilde V(x)$ and 
\begin{equation}\label{paraminfbouncy}
	\tilde\lambda(x,v) :=  (v \cdot \nabla_x \tilde U)_+,  \quad R_x^{\tilde U} := v - \frac {2  (v \cdot  \nabla_x \tilde U)}{\| \nabla_x \tilde U \|^2}  \nabla_x \tilde U.
\end{equation}
As customary, in the above $c_+$ denotes the positive part of $c \in \R$ and ${\cdot}$ is just scalar product in $\R^d$. For future use we also set 
\be\label{defa}
a(x,v) = (v \cdot \nabla_x \tilde U) \,.
\ee
The rate of decay to equilibrium in $\hat{\ltm}:= \{f \in \ltm : \int f d\mu =0\}$ for the dynamics generated by \eqref{genhampdmp} can  be studied via hypocoercive techniques; this has been done in detail in \cite{Andrieu} (see also references therein). Using \cite[Theorem 1]{Andrieu} \footnote{The generator we are studying is a particular case of the generator $\cL_2$ defined in \cite[eqn (3)]{Andrieu}. To obtain our generator from $\cL_2$ in \cite{Andrieu} take $m_2=1, \lambda_{ref}(x)=\lambda_{ref}, F_0(x) = \nabla_x \tilde{V}, F_1(x) = \nabla_x \tilde{U}, K=1, n_1(x)= \nabla_x \tilde{U}/\|\nabla_x \tilde{U}\|, \lambda_1(x,v) = \tilde{\lambda}(x,v), (\mathcal B_1f)(x,v) = f(x,\rxu)$ and $(\Pi_v f)(x,v) = (\tilde Q f)(x,v)$. In doing so bear in mind that their $U_0$ is our $\tilde V$ and their $U$ is our $V$.} one can show the following: if the potential $V(x)$ satisfies assumption {\bf{H.1}} in \cite{Andrieu} and the auxiliary potential $\tilde{V}$ is such that 
$$
|\nabla_x \tilde V(x)| \leq b (1+ |\nabla_x V(x)| ), \quad \mbox{for some } b>0,\footnote{This is to satisfy \cite[Assumption \bf{H.2} (c)]{Andrieu}.} 
$$
then there exist constants $C, c>0$ such that
$$
\|e^{t \cL}f\|_{\ltm} \leq C e^{-ct}\|f\|_{\ltm}, \quad \mbox{for every } f \in \hat{\ltm} \,.
$$
Note that the functional framework for PDMPs is rather involved, and partly still to be developed as generators of PDMPs do not enjoy smoothing properties \cite{Andrieu, Monm}. However,  it has been shown in \cite{Andrieu} that, under the above assumptions the operator  $\cL$ generates a strongly continuous contraction semigroup in $\ltm$ and that the set $C_b^2$ (of bounded functions  with bounded derivatives up to order two) is a core for $\cD(\cL) \cap \cD(\cL^*)$, hence throughout this section we will assume that $\cL$ and $\cL^*$ act on such a set. As domain for  $\cL'$ (and all dual operators) we will instead consider the set $\dense$ of positive Schwartz functions.  

\begin{lemma}\label{lemma:dualsPDMPexample}
	The $L^2$ and $\ltm$ (formal) adjoint of the generator $\cL$ in  \eqref{genhampdmp} are given, respectively, by
	\begin{align*}
		\cL'& =  \tlh'+\cL_R'+\cL_Q'\\
		\cL^* &= \tlh^*+\cL_R^*+\cL_Q^*\, ,
	\end{align*}
	where
	\begin{align}
		(\cL'_Q f)(x,v) & = \lambda_r\pi_0(v) \left(\int f(x,w) dw\right) - \lambda_r f(x,v)\,, \label{def:lQprime} \\
		(\cL_R' f)(x,v) &= \tilde \lambda (x,v) [f(x,R_x^{\tilde U} v) - f(x,v)] - ( v \cdot \nabla_x \tilde U) f(x,R_x^{\tilde U} v) \nonumber \\
		& = ( -  (v \cdot \nabla_x \tilde U))_+ f(x,R_x^{\tilde U} v) - \tilde\lambda(x,v) f(x,v)\, \nonumber 
	\end{align}
	and 
	$$
	\tlh'=-\tlh\,, \quad \cL_R^*=\cL_R'\,.
	$$
	Moreover, 
	$$
	\cL_Q^*=\cL_Q \quad \mbox{and} \quad \tilde\cL^*_Hf = -\tilde\cL_H f+  (v \cdot \nabla_x\tilde U) f \,.
	$$
\end{lemma}
\begin{proof}[Proof of Lemma \ref{lemma:dualsPDMPexample}] We show how to calculate $\cL_R'$, the rest is obtained with similar calculations:
	\begin{align*}
		\int dx \, dv  (\cL_R f) (x,v) g(x,v) & = 
		\int dx \, dv \,  \tilde \lambda (x,v) [f(x,R_x^{\tilde U} v) - f(x,v)] g(x,v) \,.
	\end{align*}
	Now make change of variable  $v' = \rxu v$ and then observe that  transformation $\rxu v$ is involutive,  so that $\rxu v'= v$,  and the Jacobian of the transformation is -1. We will also use below the identity 
	$$
	a(x,R_x^{\tilde U} v) = ( R_x^{\tilde U} v \cdot \nabla_x \tilde U ) = - ( v, \nabla_x \tilde U ) = -a(x,v) \,.$$
	With this in mind, 
	\begin{align*}
		\int \!\!\!dx \, dv  (\cL_R f) (x,v) g(x,v) & = \!\!
		\int \tilde \lambda (x,v) f(x,R_x^{\tilde U} v) g(x,v) dx\, dv - \int  \tilde \lambda (x,v) f(x,v) g(x,v) dx\, dv  \\
		& = \!\! \int \tilde \lambda (x, R_x^{\tilde U} v') 
		f (x,v') g(x,R_x^{\tilde U} v') dx dv'- \int  \tilde \lambda (x,v) f(x,v) g(x,v) dx\, dv \\
		& =\!\! \int (- ( v \cdot \nabla_x \tilde U ))_+ f(x,v) g(x,R_x^{\tilde U} v) dx dv- \int  \tilde \lambda (x,v) f(x,v) g(x,v) dx\, dv \,.
	\end{align*}
	This gives the second expression for  $\cL_R'$. The first is found by observing that $(-c)_+=c_-= c_+-c$, for every $ c \in \R$. To find the expression for $\cL_R^*$ one also needs to use 
	\be\label{v2conserved}
	\left\vert v\right\vert^2 = \left\vert \rxu v \right\vert^2, \quad \mbox{for every } v \in \R^d.
	\ee
\end{proof}
Using the above lemma, the symmetric and antisymmetric parts (in $\ltm$) of the operator \eqref{genhampdmp} are
\begin{align}\label{defLS}
	(\cL_S f)(x,v)&= \frac{\cL + \cL^*}{2}f (x,v) \nonumber \\
	&= (\cL_Q f)(x,v)+ 
	\frac{1}{2} \left( ( a_{-}(x,v))+ a_+(x,v) \right) \left[ f(x, R_x^{\tilde U} v) -f(x,v) \right] 
\end{align}
and 
\be\label{la}
(\cL_A f)(x,v)= \frac{\cL - \cL^*}{2}f (x,v) = 
(\tilde\cL_H f)(x,v)+ \cL_{JA}f  \,, 
\ee
where 
\be\label{lja}
\cL_{JA}f:=\frac{1}{2}a(x,v)\left[ f(x,R_x^{\tilde U}v) -f(x,v)\right] \,,
\ee
and, for the sake of clarity, we iterate that $a_+(x,v)$ and $a_-(x,v)$ are, respectively, the positive and negative parts of the function $a(x,v)$. 
We now come to show that the operator $\cL'= \cL_A'+\cL_S'$ can be put in pre-GENERIC form \eqref{pre-GENERIC}. To this end,  let 
\be\label{tildeS}
\cS_{\hV}(f) =  \iint dx dv \,  f(x,v) \log f(x,v) + \iint dxdv \,  f(x,v) \left(\hV (x) + \frac{|v|^2}{2} \right) \,,  
\ee
and note that, when $\hV (x) = V(x)$ the above entropy functional coincides with $\cS_{\mu}$. Then the following holds. 
\begin{proposition}\label{prop:hamPDMP}
	Let $\cL$ be the generator of the Hamiltonian PD-MCMC introduced in \eqref{genhampdmp}. Then, for any potential $\hV = \hV(x)$ such that the second addend of \eqref{tildeS} is integrable,  the operator $\cL'$ can be written in the form \be\label{pre-GENERICforPDMP}
	\cL'\rho = \tilde{W} (\rho) + \vd_{\xi}\tilde{\psis}\left(\rho;- \frac{\vd \cS_{\hV}(\rho)}{2}\right)
	\ee                     
	where $\tilde{W} (\rho)=\cL_A'\rho$ and, for every $\rho \in \dense$, $\tilde{\psis} \left(\rho; \xi\right) = \psis\left(\rho; \xi\right)+\psisj\left(\rho; \xi\right)$ with $\psis$ as in \eqref{psiforLQ} and $\psisj$ given by
	\be\label{disspotentialLSJ} 
	\psisj(\rho; \xi) := \frac 12
	\iint dx \, dv \, (a_++a_-)(x,v) 
	\sqrt{\rho(x,\rxv)\rho(x,v)} \left[ \cosh{\left(\xi(x,\rxv)-\xi(x,v)\right)} -1 \right] \,,
	\ee
	is positive and convex (in the the second argument).
\end{proposition}
Before moving on to the proof, let us point out that the above proposition is not making any statements about the degeneracy condition \eqref{degconpre-GENERIC}; we will comment on the validity (rather, non validity) of such a condition in Note \ref{noteonfuncinequality} below.  
\begin{proof}[Proof of Proposition \ref{prop:hamPDMP}]
	Let us start by decomposing the operator $\cL_S$:
	$$
	\cL_S f= \cL_Q f + \cL_{JS}f
	$$
	where
	$$
	(\cL_{JS}\rho)(x,v) = \frac 12 (a(x,v)_++a(x,v)_-)
	\left[\rho(x, \rxv)- \rho(x,v)\right]\,.
	$$
	Therefore  $\cL_S'=\cL_Q'+ \cL_{JS}$ as $\cL_{JS}'=\cL_{JS}$;  then, by \eqref{psiforLQ} and Note \ref{note:any potentialandersthermost}, to show \eqref{pre-GENERICforPDMP} one just needs to observe the following
	\be\label{LJSgradientform}
	(\cL_{JS}f)(x,v) = \vd_{\xi}\hat{\psis}\left(f; -\frac{\vd \cS_{\hV}(f)}{2}\right)\,.
	\ee
	To see that the above holds, let us first calculate the variational derivative of $\hat{\psis}$: for every $\eta \in \dense$,  
	\begin{align*}
		& \left(\vd_\xi\psisj(\rho;\xi),\eta\right) = \frac{d}{dt}\psisj(\rho;\xi+t\eta)\Big\vert_{t=0}
		\\
		=& \frac 12 \int\!\!\!\!\!\int (a_++a_-)(x,v)\sqrt{\rho(x,R^{\U}_xv)\rho(x,v)}\sinh
		\Big(\xi(x,R^{\U}_x v)-\xi(x,v)\Big)\Big(\eta(x,R^{\U}_x v)-\eta(x,v)\Big)dxdv.
	\end{align*}
	We want to calculate the above expression at 
	$\xi = -\frac 12\vd \cS_{\hV}(\rho)$. In particular we need to calculate the difference
	$$
	\xi(x,R^{\U}_x v)-\xi(x,v) = - \frac 12 \left( \vd \cS_{\hV}(\rho)(x,v) - \vd\cS(\rho)(x,v)\right) \, . 
	$$
	It is in doing so that one can see that the specific choice of potential $\hV$ does not matter, as the potential is cancelled in the difference. Indeed, since
	$$
	\vd \cS_{\hV}(\rho)(x,v)=\log \rho(x,v)+\hV(x)+\frac{|v|^2}{2}+1 \, ,
	$$
	from \eqref{v2conserved}, we have 
	$$
	-\frac 12 \left( \vd \cS_{\hV}(\rho)(x,R^{\U}_x v) - \vd \cS_{\hV}(\rho)(x,v) \right) =\log \left(\sqrt{\frac {\rho(x,v)} {\rho(x,R^{\U}_x v)} }\right) \,.
	$$
	Therefore, using the identity  $(a_++a_-)(x,v)=(a_++a_-)(x, R^{\U}_x v)$, we obtain
	\begin{align*}
		&\left( \vd_\xi\psisj(\rho;-\frac{1}{2}
		\vd \cS(\rho)),\eta\right)
		\\&=\frac{1}{4}\int (a_++a_-)(x,v)\rho(x,v)\Big(\eta(x,R^{\U}_x v)-\eta(x,v)\Big)\,dxdv
		\\&\quad-\frac{1}{4}\int (a_++a_-)(x,v)\rho(x,R^{\U}_x v)\Big(\eta(x,R^{\U}_x v)-\eta(x,v)\Big)\,dxdv
		\\&= \frac 12 \int (a_++a_-)(x,v)\rho(x,v)\Big(\eta(x,R^{\U}_x v)-\eta(x,v)\Big)\,dxdv
		\\&= \frac 12 \int (a_++a_-)(x,v)\Big(\rho(x,R^{\U}_x v)-\rho(x,v)\Big)\eta(x,v)\,dxdv \,,
	\end{align*}
	which is the desired result. 
\end{proof}    

\smallskip \noindent
{\bf Relation to Section \ref{section:decomposition}. }
In Proposition \ref{prop:hamPDMP} the functional $\tilde{\psis}$ appeared as an ansatz. We show here that  $\tilde{\psis}$ does indeed coincide with the dissipation potential that we would obtain using the procedure of Section \ref{section:decomposition},  Theorem \ref{theorem:decomposition} in particular. First calculate the Hamiltonian 
\begin{align*}
	\mathfrak H_s (\rho;\xi) & := \iint e^{-\xi (x,v)} (\cL_S e^{\xi})(x,v) \rho(x,v) dx\, dv\\
	& = \iint e^{-\xi (x,v)} (\cL_Q e^{\xi})(x,v) \rho(x,v) dx\, dv + 
	\iint e^{-\xi (x,v)} (\cL_{JS} e^{\xi})(x,v) \rho(x,v) dx\, dv \,.
\end{align*}
For general purposes, for every $\rho \in \dense$, one can take $\cD_{\rho}(\mathfrak{H}_s)$ to be the set of smooth functions which grow at most quadratically in $v$ and at most like $V(x)$ in the space variable $x$. 
We work on the second addend of the above, which will give $\psisj$; applying a similar procedure to the first addend gives $\psis$. 
\begin{align*}
	\hat{\mathfrak H}_s (\rho;\xi) &:=
	\iint e^{-\xi (x,v)} (\cL_{JS} e^{\xi})(x,v) \rho(x,v) dx\, dv  \\
	& = \iint e^{-\xi (x,v)}
	\frac{1}{2} \left[a_+(x,v)+a_{-}(x,v) \right]
	\left[e^{\xi(x, \rxv)}-e^{\xi(x,v)} \right] \rho(x,v) dx\, dv\,.
\end{align*}
We want to show that \eqref{eq: Psi from Ls} holds for this example, i.e. that the following identity holds:
\begin{align*}
	\psisj(\rho;\xi) &= \hat{\mathfrak H}_s \left(\rho;\xi + \frac12 \vd{\cS_{\hV}}(f)\right) 
	- \hat{\mathfrak H}_s \left(\rho; \frac12 \vd{\cS_{\hV}}(\rho)\right) \\
	& = \iint \frac12 \left[a_+(x,v)+a_{-}(x,v) \right]
	\left[e^{\xi(x, \rxv)- \xi(x,v)} -1 \right] \sqrt{\rho(x, v) \rho(x, \rxv)} dx\, dv\,.
\end{align*}
Indeed, by observing that the above integral   does not change under the change of variable $v'= \rxv$, one can easily see that the above does indeed coincide with the expression for $\psisj$ given in \eqref{disspotentialLSJ}. 

\smallskip \noindent

{\bf Relative Entropy dissipation for Hamiltonian PD-MCMC. } We now want to show that the entropy functional $\cS_{\mu}$ decays along the flow, so from now on we pick $\hV(x) = V(x)$ in the representation \eqref{pre-GENERICforPDMP}. 
\begin{Note}\label{noteonfuncinequality}
	Within the pre-GENERIC framework, entropy decay is implied by either the degeneracy condition \eqref{degconpre-GENERIC} or  the dissipation condition \eqref{dissipcond}. 
	In our case the latter condition  would read
	\be\label{degconPDMP}
	(\cL_A'f, \vd\cS_{\mu} (f)) \leq 0 \,.
	\ee
	Let us calculate explicitly the LHS of the above.
	The operator $\cL_A'$, i.e. the $L^2$ adjoint of the operator $\cL_A$ defined in \eqref{la},  is
	$$
	(\cL_A' f)(x,v)= - \tlh f + \cL_{JA}'
	$$
	where
	$$
	\cL_{JA}' = - \frac{1}{2} a (x,v)\left[ f(x, R_x^{\tilde U}v)+ f(x,v) \right] \,.
	$$
	Because 
	$$
	- (\tlh f, \vd\cS_{\mu}(f)) = \int a(x,v) f(x,v) \, dx\, dv \, ,
	$$
	and  
	$$
	(\cL_{JA}'f, \vd\cS_{\mu}(f)) = \int (\cL_{JA}'f)(x,v) \log f (x,v) dx\, dv,  
	$$
	(the above coming from the fact that $\cL_{JA}(V+\frac{v^2}{2}+1)=0$, where we recall that $\cL_{JA}$ is defined in \eqref{lja})
	one has 
	$$
	\left( \vd\cS_{\mu} (f), \cL_A' f\right) = \int a(x,v) f(x,v) \, dx\, dv +
	\int (\cL_{JA}'f)(x,v) \log f (x,v) dx\, dv \,,
	$$
	or, more explicitly, 
	\begin{align}
		\left( \vd\cS_{\mu} (f), \cL_A' f\right) 
		& = \int a(x,v) f(x,v) \, dx\, dv \, +
		\frac12 \int a(x,v) f(x,v) \log\frac{ f(x, \rxv)}{ f(x,v)} dx\, dv \label{ed1}\\
		& = \int a(x,v) f(x,v) \, dx\, dv \, - \frac12  \int a(x,v) [f(x, \rxv)+f(x,v)] \log f(x,v) dx\, dv \,. \label{ed2}
	\end{align}
	\textup{We can now make some observations on \eqref{degconPDMP}
		\begin{itemize}
			\item Using the expression \eqref{ed1}, we can see that the inequality \eqref{degconPDMP} does not hold as a functional inequality, in the sense that it does not hold for every $f$ (in an appropriate class, to which the counterexample which we are about to exhibit should belong). This is easy to see in one dimension:  if $d=1$ then $\rxu v = -v$; if we choose $f$ to be in product form, more specifically, $f(x,v) = e^{-\bar{V}(x)} \pi_c(v)$ where $\pi_c(v)$ is a Gaussian with mean $c \in \R$ and variance $\sigma$, $\bar{V}$ is such that $e^{-\bar V}$ is normalised in space and $V, \tilde{V}$ such that $\nabla_x\tilde{U}(x)=1$, then $a(x,v)=v$ and
			$$
			\left( \vd\cS_{\mu} (f), \cL_A' f\right)  = -\frac{c^3}{\sigma} \,,
			$$
			which can  have any sign. 
			We note that in order to show entropy decay one needs not prove \eqref{degconPDMP} for every $f$, it is sufficient to prove \eqref{degconPDMP} only along the flow, but we have not been able to do so. 
			\item Finally, we observe that  \eqref{degconPDMP} holds (with equality) for every function $f$ which is reflection invariant, i.e. such that $f(x, \rxv) = f(x,v)$. Indeed if this is the case then the second addend in \eqref{ed1} vanishes. For the first, notice that by a change of variables
			$$
			\int a(x,v) f(x,v) dx\, dv =  - \int a(x,v) f(x, \rxv) dx\, dv\, ,
			$$
			which implies
			$$
			\int a(x,v) \left( f(x,v) + f(x, \rxv)\right) dx\, dv = 0 \quad \mbox{ for every}~ f; 
			$$
			by writing $f(x,v)$ as $f(x,v) = (f(x,v)+f(x, \rxv))/2 + (f(x,v)-f(x, \rxv))/2$ one then has
			\be\label{halfint}
			\int a(x,v) f(x,v) dx\, dv = \int a(x,v) \frac{f(x,v) - f(x, \rxv)}{2} dx\, dv \quad \mbox{ for every}~f \,.
			\ee
			This suggests that if the flow preserves the invariance of the profile $f$ under reflections then \eqref{degconPDMP} is true along the flow, assuming $f_0$ is reflection invariant. 
		\end{itemize}
	}
\end{Note}

In view of these observations, we prove entropy decay in a different way. 
\begin{proposition}
	Consider the Hamiltonian PD-MCMC with generator $\cL$ as in \eqref{genhampdmp} and let $\mu$ be the measure on $\R^{2d}$ defined in \eqref{invmeas multid}.  Then the relative entropy $\cS_{\mu}$  decreases along the solution of the FP equation $\pa_tf_t(x,v)= (\cL'f_t) (x,v)$.  
\end{proposition}
\begin{proof}
	In order to prove entropy decay we need to show the following
	\begin{align*}
		\frac{d\cS_{\mu} (f_t)}{dt} & =  \left( \vd\cS_{\mu} (f_t), \cL_S' f_t\right)+ \left( \vd\cS_{\mu} (f_t), \cL_A' f_t\right) \leq 0.
	\end{align*}

	Because $\cL_S' = \cL_Q'+ \cL_{JS}$ from \eqref{psiforLQ} and \eqref{LJSgradientform} we already know 
	$$
	\left( \vd\cS_{\mu}, \cL_S' f\right) \leq 0 \quad \mbox{for all } f;
	$$
	more specifically, thanks to the convexity of the functionals $\psis$ and $\hat\psis$,  we know that both 
	\be\label{knowneg1}
	\left( \vd\cS_{\mu} (f), \cL_Q' f\right) \leq 0
	\ee
	and
	\be\label{knowneg2}
	\left( \vd\cS_{\mu} (f), \cL_{JS} f\right) \leq 0\,.
	\ee
	Hence, we know 
	$$
	\left(\vd\cS_{\mu} (f), \cL_Q' f\right)  = \int \log f (\cL_Q'f) dx\, dv +
	\lambda_r \int f\left( 1- \frac{v^2}{2}\right) dx\, dv \leq 0 
	$$
	and 
	\begin{align*}
		\left( \vd\cS_{\mu} (f), \cL_{JS} f\right)  & = \frac{1}{2} \int  |a(x,v)|\, \log f \left[f(x, \rxv) - f(x,v) \right]dx\, dv\\
		& =  \frac{1}{2} \int  |a(x,v)|\, \log\left( \frac{f(x,\rxv)}{f(x,v)}\right) f(x,v)dx\, dv\leq 0 
	\end{align*}
	and we want to prove the inequality 
	\begin{align}
		& \int \log f (\cL_Q'f) dx\, dv +
		\lambda_r \int f\left( 1- \frac{v^2}{2}\right) dx\, dv  +
		\frac{1}{2} \int  |a(x,v)|\, \log\left( \frac{f(x,\rxv)}{f(x,v)}\right) f(x,v) dx\, dv \nonumber \\
		& + \int a(x,v) f(x,v) dx\, dv +
		\frac12 \int a(x,v) f(x,v) \log\frac{ f(x, \rxv)}{ f(x,v)} dx\, dv \leq 0 \label{conj2} \,.
	\end{align}
	In view of \eqref{knowneg1} and \eqref{knowneg2},  it is sufficient to prove the following inequality, namely
	\begin{align}
		&\frac{1}{2} \int  |a(x,v)|\, \log\left( \frac{f(x,\rxv)}{f(x,v)}\right) f(x,v) dx\, dv+ \int a(x,v) f(x,v)  dx\, dv\notag
		\\&\qquad+
		\frac12 \int a(x,v) f(x,v) \log\frac{ f(x, \rxv)}{ f(x,v)}dx\, dv \leq 0 \,.
	\end{align}
	To this end start by observing that, using \eqref{halfint}, the above is equivalent to
	\begin{align}\label{wtp}
		&\frac{1}{2} \int  |a(x,v)|\, \log\left( \frac{f(x,\rxv)}{f(x,v)}\right) f(x,v) dx\, dv\nonumber\\
		&+ \frac12 \int a(x,v) \left[ f(x,v) - f(x,\rxu) \right] dx\, dv +
		\frac12 \int a(x,v) f(x,v) \log\frac{ f(x, \rxv)}{ f(x,v)} dx\, dv \leq 0 \,.
	\end{align}
	In turn, by writing $a=a_+-a_-$ and $|a|=a_++a_-$ and observing that for every function $g$ one has $$\int a_-(x,v) g(x,v) = \int a_+(x,v) g(x,\rxu),$$
	\eqref{wtp} is equivalent to
	\be\label{wtp1}
	\int a_+(x,v) \left[ f(x,v) - f(x,\rxu) \right] dx\, dv +
	\int a_+(x,v) f(x,v) \log\frac{ f(x, \rxv)}{ f(x,v)} dx\, dv \leq 0\,.
	\ee
	The above is now easy to prove; indeed, since $\log u \leq u-1$, we have
	$$
	\int a_+(x,v) f(x,v) \log\frac{ f(x, \rxv)}{ f(x,v)} dx\, dv\leq \int a_+(x,v) 
	\left[ f(x,\rxu) - f(x,v)
	\right]dx\, dv\,
	$$
	which is precisely \eqref{wtp1}.
\end{proof}
\appendix
\section{Basic facts about large deviations}\label{appendixA}
The notation in this appendix is independent of the notation in the rest of the paper. 
Let $(X, d)$ be a metric space,  $\{\chi_n\}_{n \in \mathbb N}$ be a sequence of $X$-valued random variables and $\cI: X \rightarrow  [0, \infty]$ be a lower semicontinuous function. The sequence $\{\chi_n\}_{n \in \mathbb N}$ satisfies a Large Deviation Principle (LDP) with rate function  $\cI$ if the following two conditions are satisfied (see \cite[Chapter 1]{FengKurtz}): 
\begin{itemize}
	\item for every open set $A \subseteq X$
	$$
	\liminf_{n \rightarrow \infty} \frac{1}{n} \log \mathbb P(\chi_n \in A) \geq -  \inf_{x \in A} \cI(x) 
	$$
	\item for every closed set $B \subseteq X$
	$$
	\limsup_{n \rightarrow \infty} \frac{1}{n} \log \mathbb P(\chi_n \in A) \leq -  \inf_{x \in B} \cI(x) \,.
	$$
\end{itemize}
The rate function $\cI$ is said to be a {\em good rate function} if for every $a \in [0, \infty)$ the set $\{ x: \cI(x)\leq a\}$ is compact. A short-hand notation to express the above is 
$$
\mathrm{Prob}\Big( \chi_n\approx \chi\Big)\overset{n\rightarrow \infty}{\sim} e^{-n \cI(\chi)} \,. 
$$
\section{Some observations about the definition of pre-GENERIC}\label{app:equiv}

For completeness we show here that the definition \eqref{pre-GENERIC} of pre-GENERIC is equivalent to \eqref{equivdefpre-generic}. To this end let us consider the function 
$$
\varphi(\rho; g):= \varphi_{\rho}(g) = \psi_{\rho}(g-W(\rho)) + \psi^{\star}_{\rho}\left( -\frac12 \vd \cS(\rho)\right) + \frac12 \left( g, \vd \cS(\rho)\right) \,.
$$
We look at the above as a function of $g$, for every $\rho$ fixed and we want to prove that $\varphi_z(g)=0$ iff $g = W(z) +(\vd_{\xi} \psi^{\star}_z)\left( -\frac 12 \vd \cS(z)\right)$. To this end note that the function $\varphi$ is the sum of two strictly convex functions plus a linear function and it is therefore  strictly convex (in $g$). Moreover, by the Young-Fenchel inequality we have 
$$
\frac12 \left( g, \vd \cS(\rho)\right) \leq  \psi_{\rho}(g) + \psi^{\star}_{\rho}\left( \frac12 \vd \cS(\rho)\right), 
$$ 
so that, using the symmetry of $\psi^{\star}$ and the orthogonality condition \eqref{degconpre-GENERIC}, $\varphi$ is positive as well and it attains its minimum (zero) iff 
$ \vd_g\varphi_{\rho}(g)=0$. Now, 
\begin{align*}
	\vd_v\varphi_{\rho}(g)& = \vd \psi_{\rho}(g-W(\rho)) + \frac 12 \vd \cS(\rho) =0 	\\
	& \Leftrightarrow g = W(\rho) +(\vd_g \psi_{\rho})^{-1}\left( -\frac 12 \vd \cS(\rho)\right)\\
	& \Leftrightarrow  g = W(\rho) +(\vd_{\xi} \psi^{\star}_{\rho})\left( -\frac 12 \vd \cS(\rho)\right)\, ,
\end{align*}
having used the fact that $\psi_{\rho}=\psi_{\rho}(g)$ and $\psi^{\star}_{\rho}= \psi^{\star}_{\rho}(\xi)$ are Legendre dual of each other, hence
$(\vd_v \psi_{\rho})^{-1} = \vd_{\xi} \psi^{\star}_{\rho}$. 
\begin{Note}In absence of the orthogonality condition \eqref{degconpre-GENERIC} one can observe that the following condition holds 
	$$
	\psi_{\rho}(v-W(\rho)) + \psi^{\star}_{\rho}\left( -\frac12 \vd \cS(\rho)\right) + \frac12 \left( v-W(\rho), \vd \cS(\rho)\right)=0
	$$
	if and only if $v = W(\rho) +(\vd_{\xi} \psi^{\star}_{\rho})\left( -\frac 12 \vd \cS(\rho)\right)$. 
	To see this,  just consider the function 
	$$
	\tilde{\varphi}_{\rho}(g)= \psi_{\rho}(g-W(\rho)) + \psi^{\star}_{\rho}\left( -\frac12 \vd \cS(\rho)\right) + \frac12 \left( g - W(\rho), \vd \cS(\rho)\right),  
	$$	
	which is again convex and positive by Young-Fenchel inequality.  The rest of the reasoning is identical to the one done above. 
\end{Note}

\section*{Acknowledgements} The research of M.H. Duong was supported by the EPSRC Grant EP/V038516/1.

\end{document}